\documentclass[a4paper,12pt,reqno]{amsart}
\usepackage{color}
\usepackage{a4wide}
\usepackage{geometry}
\usepackage{hyperref}
\usepackage{amsmath}
\usepackage{amssymb}
\usepackage{amsthm}
\usepackage[active]{srcltx}
\usepackage[dvips]{graphicx}

\numberwithin{equation}{section}

\def\max{{\rm max}}

\def\Sym1{\mathcal{S}(1)}

\def\cG{{\mathcal G}}

\def\cH{{\mathcal H}}
\def\cK{{\mathcal K}}

\def\cD{{\mathcal D}}

\def\cM{{\mathcal M}}

\def\cP{{\mathcal P}}

\def\cN{{\mathcal N}}

\def\cQ{{\mathcal Q}}

\def\cS{{\mathcal S}}

\newcommand{\note}[1]{\marginpar{\tiny\emph{#1}}}

\newtheorem{thm}{\textbf{Theorem}}[section]
\newtheorem{lem}[thm]{\textbf{Lemma}}
\newtheorem{prop}[thm]{\textbf{Proposition}}
\theoremstyle{remark}
\newtheorem{rem}[thm]{\textbf{Remark}}

\newtheorem{exe}[thm]{\textbf{Example}}

\theoremstyle{definition}
\newtheorem{defn}[thm]{{Definition}}

\newtheoremstyle{Claim}{}{}{\itshape}{}{\itshape\bfseries}{:}{ }{#1}
\theoremstyle{Claim}

\newcommand{\R}{{\mathbb R}}
\newcommand{\CF}{{\mathbb C}}
\newcommand{\HF}{{\mathbb H}}
\newcommand{\G}{{\mathbb G}}

\newcommand{\veps}{\varepsilon}

\newcommand{\cJ}{\mathcal{J}}
\newcommand{\cL}{\mathcal{L}}

\newcommand{\cF}{\mathcal{F}}
\newcommand{\J}{\mathcal{J}}
\newcommand{\C}{\mathcal{C}}
\newcommand{\F}{\mathcal{F}}

\newcommand{\wt}{\widetilde}

\newcommand{\Ss}{\mathcal{S}}

\newcommand{\USC}{\mathrm{USC}}
\newcommand{\LSC}{\mathrm{LSC}}

\newcommand{\Int}{\mathrm{Int}}
\newcommand{\SA}{\mathrm{SA}}

\makeatletter
\newcommand{\tpitchfork}{%
	\vbox{
		\baselineskip\z@skip
		\lineskip-.52ex
		\lineskiplimit\maxdimen
		\m@th
		\ialign{##\crcr\hidewidth\smash{$-$}\hidewidth\crcr$\pitchfork$\crcr}
	}%
}
\makeatother

\newcommand{\changelocaltocdepth}[1]{%
	\addtocontents{toc}{\protect\setcounter{tocdepth}{#1}}%
	\setcounter{tocdepth}{#1}%
}

\title[Nonlinear potential theory and PDEs]{Interplay between nonlinear potential theory and fully nonlinear elliptic PDEs}
\author[F.R. Harvey]{F. Reese Harvey}
\address{Department of Mathematics\\ Rice University\\ P.O. Box 1892\\ Houston, TX 77005-1892, USA}
\email{harvey@rice.edu (F. Reese Harvey)}
\author[K.R. Payne]{Kevin R. Payne}
\address{Dipartimento di Matematica ``F. Enriques''\\ Universit\`a di Milano\\ Via C. Saldini 50\\ 20133--Milano, Italy}
\email{kevin.payne@unimi.it (Kevin R.\ Payne)}

\date{\today} 


\begin{document}

\dedicatory{ Dedicated to Blaine Lawson on the occasion of his 80th birthday}        

\begin{abstract}
	
	We discuss one of the many topics that illustrate the interaction of Blaine Lawson's deep geometric and analytic insights. The first author is extremely grateful to have had the pleasure of collaborating with Blaine over many enjoyable years. The topic to be discussed concerns the fruitful interplay between {\em nonlinear potential theory}; that is, the study of subharmonics with respect to a general constraint set in the $2$-jet bundle and the study of subsolutions and supersolutions of a nonlinear (degenerate) elliptic PDE. The main results include (but are not limited to) the validity of the comparison principle and the existence and uniqueness to solutions to the relevant Dirichlet problems on domains which are suitably ``pseudoconvex''. The methods employed are geometric and flexible as well as being very general on the potential theory side, which is interesting in its own right. Moreover, in many important geometric contexts no natutral operator may be present. On the other hand, the potential theoretic approach can yield results on the PDE side in terms of non standard structual conditions on a given differential operator.
	
\end{abstract}

\maketitle

\makeatletter
\def\l@subsection{\@tocline{2}{0pt}{2.5pc}{5pc}{}}
\makeatother

 \setcounter{tocdepth}{1}
\tableofcontents

\changelocaltocdepth{2}

\section{Introduction}\label{sec:intro} 

Our main aim is to give a partial survey of a research endeavor which was initiated in a trio of papers of Harvey and Lawson \cite{HL09a}, \cite{HL09b} \and \cite{HL09c} published in 2009 that has grown into a wide ranging investigation with many interesting and important avenues still to pursue. For simplcity of the exposition and in order to make the discussion more accessible to analysts, we will focus on the Euclidian setting of open subsets $X$ of $\R^n$, although $X$ could also be a Riemannian manifold as in \cite{HL11a} and \cite{HL13b}, or an almost complex manifold as in \cite{HL15}. We will emphasise the fruitful interplay between {\em nonlinear potential theory}; that is, the study of the family of {\em $\cF$-subharmonics} with respect to a given {\em subequation (constraint set)}
\begin{equation}\label{subequation}
\cF \subset \cJ^2(X) := X \times \cJ^2 := X \times \R \times \R^n \times \cS(n), \ \ \ X \subset \R^n 
\end{equation}
and the study of solutions/subsolutions/supersolutions of a given {\em fully nonlinear (elliptic) PDE}
\begin{equation}\label{operator}
F(x, J^2_xu)  := F(x,u(x),Du(x),D^2u(x)) = 0, \ \ x \in X \subset \R^n
\end{equation}
determined by a given {\em operator} $F \in C(\cJ^2(X))$. The equation \eqref{operator} will also be written more succinctly as 
\begin{equation}\label{equation}
F(J^2u) = 0 \ \ \text{on} \ \ X.
\end{equation}
Here $\cJ^2$ is the vector space of {\em $2$-jets}. We will use the notation $J^2_xu$ for the second order Taylor development of $u$ indifferently with respect to the differentiability of $u$. This interplay has been developed in detail in \cite{CP17}, \cite{CP21}, \cite{CHLP21} and \cite{CPR21}.  

Given the fully nonlinear setting, one cannot expect solutions to be regular in general, and distribution theory is generally available for convex subequations or equations in divergence form. Hence all notions are to be  interpreted pointwise in the {\em viscosity sense} that will be recalled in Definition \ref{defn:Fsub} (see \cite{HL14c} for the equivalence of the distributional approach and the viscosity approach in the convex case). 

There is a satisfying unification that comes from a potential theoretic (pluripotential theoretic) viewpoint as it includes classical (Laplacian) subharmonics, convex and quasiconvex functions as well as new geometric potential theories (some of which are useful for theoretical physics) as well as an immense universe of first and second order potential theories determined by classes of (degenerate) elliptic operators. 

We now descibe the main motivating principles. There are many opportunities for cross-fertilization and synergy between the potential theory and the operator theory. First, the conditions imposed on a constraint set $\cF$ correspond to and encode structural conditions on the operator $F$; for instance, a convex constraint set $\cF$ corresponds to a concave operator $F$. Second, the subequation $\cF$ ``frees'' a given PDE from any particular form of $F$ (many different $F$ correspond to the same $\cF$); this is an important point in the work of Krylov \cite{Kv95} on the general notion of ellipticity. Moreover, $\cF$ ``liberates'' the user from needing an operator $F$ to apply nonlinear elliptic potential theory. Third, ``forgetting'' about the operator leads to interesting questions that at first glance might not seem important for operator theory and provides a ``machine'' for formulating new conjectures and theorems. For instance, taking one's cue from known results in pluripotential theory or convex analyisis, one is led to seek generalizations in other potential theoretic situations as well. Some examples of this will be discussed in subsection \ref{sec:SCV} below. In this way, one can find ``welcome surprises'' in the operator theory. Fourth, along with a rich abundance of geometrically motivated potential theories, there are many new PDEs to discover. For example, as will be discussed in subsection \ref{sec:GPT}, while every calibrated geometry has an underlying potential theory, known ``natural'' smooth operators are ``rare gems''.  
Nevertheless, for any given subequation $\cF$, one can construct several ``non smooth'' operators. One good example is the construction of the so-called {\em canonical operator} associated to a given subequation. This is a canonical construction that is scattered out, first in \cite[Remark 14.11]{HL11a}  and \cite[Examples 3.4 and 3.5]{HL10} and then better explained in \cite[subsection on canonical operators in section 6]{HL19a} and \cite[Proposition 11.17]{CHLP21}. It includes the {\em truncated Laplacians}, among the many examples, so it might be refered to as the ``cannonical eigenvalue operator construction''.  These operators are discussed a bit further in Example \ref{exe:COs}. Another good example is the {\em signed distance operator} 
\begin{equation}\label{SDO}
F(x,J):=  \left\{ \begin{array}{cl} {\rm dist}(J, \partial \cF_x) & J \in \cF_x \\
-{\rm dist}(J, \partial \cF_x) & J \in \cJ^2 \setminus \cF_x \end{array}\right. ,
\end{equation}
where 
\begin{equation}\label{fiber}
	\cF_x := \{ J \in \cJ^2: \ (x,J) \in \cF\}
\end{equation}
is the {\em fiber of $\cF$ over $x \in X$}. The operator \eqref{SDO} was studied in the pure second order case in Theorem 3.2 of \cite{Kv95}. Finally, in the rare cases when a natural operator $F$ is known for a fixed $\cF$-potential theory,  the operator $F$ will have much to say about the potential theory; for example, by taking derivatives of the equation. 

Having stated the main aims and philosohical motivations, we proceed to describe the origins and objectives, along with key concepts, nice features, some results and significant examples which illustrate the theory. We begin with discussion of the origins of the investigation which led to a hierarchy of potential theories. 

\subsection{Potential theories: from calibrations to subequations.}\label{sec:GPT} The story begins in calibrated geometry. The geometric side of calibrated geometry was developed to emphasise the {\em calibrated submanifolds} which are those submanifolds for which the {\em calibration}  restricts to be the volume form. Said infinitesimally, a calibration $\phi$ of degree $p$ resticts to be a function on the Grassmannian of oriented $p$-planes where it attains a maximum value of one on the subset $G(\phi)$ of {\em $p$-planes calibrated by $\phi$}. In turn, a submanifold is calibrated by $\phi$ if its tangent planes are calibrated by $\phi$.

The basic example (other than calibrated geodesics), going back to Wirtinger in the last century, is the K\"{a}hler/symplectic form on $\CF^n = \R^{2n}$. Here the calibrated submanifolds are simply the complex curves in $\CF^n$. This geometric example, which has an analytic side involving a rich and well developed {\em potential theory} (or more precisely a {\em pluripotential theory}), then cries out for a pluripotential theory for other calibrations, providing the impetus to search for general $\phi$-potential theories.

The {\em $\phi$-subharmonic (or $\phi$-plurisubharmonic) functions} $u$ are easy to define for smooth functions (see \cite{HL09a}). One simply requires that the restriction of $u$ to $\phi$-submanifolds is classically subharmonic (with respect to the Laplacian $\Delta$). This imposes a constraint on the second derivative (Hessian matrix) $D^2u$ of $u$ at each point. This constraint condition is that $D^2u$ resticts to have trace zero on any $p$-plane calibrated by $\phi$. More precisely, by identifying a $p$-plane $W$ with the orthogonal projection $P_W$ onto $W$ and by using the natural inner product on the space $\cS(n)$ of second derivatives (i.e.\ the symmetric matices), the second derivative constraint set determined by $G(\phi)$ is just the {\em polar cone} 
\begin{equation}\label{polar}
	G(\phi)^{\circ}:= \{ A \in \cS(n): \ \langle A, P_W  \rangle := {\rm tr} \left( A_{|W} \right) \geq 0, \ \ \forall \, W \in G(\phi) \}
\end{equation}
of the set $G(\phi)$. This $C^{\infty}$ potential theory suffices for many purposes (see \cite{HL09a}) where it was noticed, as a (big) surprise, that the calibration plays a minor role subordinate to the set $G(\phi)$ of distinguised/calibrated $p$-planes, including the case of $G$ being the full set of $p$-planes (see \cite{HL09c}, \cite{HL11b}, \cite{HL12} and \cite{HL13c}), thus extending the realm of $\phi$-potential theories to $G$-potential theories. 

In fact, replacing $G(\phi)$ by any closed set $G \subset \cS(n)$ and then defining ``$G$-submanifolds'' and ``$G$-plurisubharmonic functions'' as above, it is still possibile to obtain a robust $G$-pluripotential theory. Again, the constraint condition on the second derivatives of a function $u$ at each $x \in X$
\begin{equation}\label{G_constraint}
\langle D^2u(x) , P_W \rangle := {\rm tr}  \left( D^2u(x)_{|W} \right) \geq 0,  \ \forall \,   W \in G;  \ \ \text{that is,} \ D^2u(x) \in 	G^{\circ},
\end{equation}
defines $G$-plurisubharmonicity of $u$ on $X$. These functions $u$ are characterized by their restrictions to $G$-submanifolds $M$ that are also minimal;  $u$ is subharmonic with respect to the induced Laplacian on $M$ (see \cite{HL14b}). Therefore it is justified to state that 
\begin{center}
	{\em $G$-pluripotential theory is the  correct pluripotential theory for the geometry of minimal $G$-submanifolds.}
\end{center}
 Note that if $G:= G(\phi)$ with $\phi$ a calibration, each $G$-submanifold is automatically actually absolutely volume minimizing (see \cite{HL82}) and hence a minimal submanifold. Potential theory based on a closed subset $G$ of the $p$-Grassmanian will be refered to as the {\em geometric case}.

This unifies many known cases and includes lots of new interesting cases. First, as our basic motivation, this includes classical pluripotential theory in complex analysis by taking $G$ to be the Grassmannian of complex lines in $\C^n$ as a subset of real $2$-planes. Second, taking $G$ to be the full Grassmannian $\G(1, \R^n)$ of real lines in $\R^n$ we include convex function theory, providing more precision to the well-known parallel between pseudoconvexity and convexity. Third, this leads to new surprising examples not related  to calibrations. As an example, for each degree $p$, let $G$ be the full Grassmannian of $p$-planes in $\R^n$. One obtains a $p$-pluripotential theory associated with the geometry of all minimal submanifolds of dimension $p$.  Another noteworthy new example is Lagrangian pluripotential theory, defined by taking $G = {\rm LAG}$ to be the set of Lagrangian $n$-planes in $\C^n$. This is the appropriate potential theory for minimal Lagrangian submanifolds \cite{HL17b}. These are the two cases where new natural polynomial operators were discovered. First, the {\em $p$-fold sum operator} whose domain is the subequation $\G(p, \R^n)^{\circ}$ (see \cite[(10.12)]{HL09c}) and second the {\em Lagrangian Monge-Amp\`{e}re operator} (see \cite[(10.11)]{HL09c} and \cite{HL17b}).

For the next level of generality, one can focus entirely on the second derivative constraint set $\cF \subset \cS(n)$, which in the geometric case is given by the polar $\cF = G^{\circ}$ of $G$. There is surprising simplicity here as well. Other than $\cF$ being closed, a single condition on $\cF$ described below, which is called {\em positivity} (P), is needed. This condition (P) ensures that the notion of {\em $\cF$-subharmonicity} for upper semicontinuous $u$ agrees with the definition $D^2u(x) \in \F$ for $C^2$ functions. Such sets $\cF$ are called {\em subequations in $\cS(n)$}. This is the pure second order constant coefficient case. This condition, besides providing the weakest possible condition ensuring {\em coherence} between the two definitions of $\cF$-subharmonicity should also be viewed as the weakest possible form of ellipticity.   

As with classical potential theory (for the Laplacian), the regularity of a general $G$-subharmonic function should only be required to be upper semi-continuous. This extension, although carried out with {\em Dirichlet duality} and {\em subaffine functions} in \cite{HL09c}, is equivalent to a viscosity theory formulation (see Remark 4.9 of \cite{HL09c}). The viscosity approach is more direct and can used easily for potential theories for subequation constraint sets with variable coefficients and dependence on all the jet variables (as will be discussed in the next paragraph). However, the (Dirichlet) duality continues to be important. It clarifies the notion of superharmonics and leads to straightforward proofs of the comparison principle. This is examined extensively in the constant coefficient setting in \cite{CHLP21}. 

As a final step in the delineation of a hierarchy of potential theories, it is natural to consider the extension from a pure second order constraint on second derivatives in the Euclidian setting of open subets $X \subset \R^n$ to the case of $X$ being a Riemannian manifold \cite{HL11a} or an almost complex manifold \cite{HL15}. In local coordinates, this constraint will have variable coefficients and may depend on the full Taylor development up to second order (not just second derivatives). Returning to the Euclidian setting, this suggests  considering subequation constraint sets $\cF \subset \cJ^2(X) = X \times \J^2$ and their associated potential theories. The needed axioms for a robust potential theory are given in Definition \ref{defn:subeq} Briefly stated, one adds two additional axioms (negativity (N) and
topological stability (T)) to the positivity (P) and requiring that $\cF$ be closed. The interplay between nonlinear potential theory and fully nonlinear elliptic PDEs most naturally takes place at this level of the hierarchy. This will be discussed further, begining in subsection \ref{sec:DP}. 

An important part of the story is that in studying this hierarchy of possible levels of potential theories; $\phi$-subharmonics, $G$-subharmonics and finally $\cF$-subharmonics, a distinguished (differential) operator $F$ is missing from the picture. This abscence of an operator has advantages (and disadvantages) which have been noted above, and will be amplified below.

\subsection{Some potential theoretic results suggested by complex analysis.}\label{sec:SCV}   A surprising number of results in complex analysis (of several variables) can be established in greater generality in nonlinear potential theory. We mention seven such topics that were suggested by results in pluripotential theory. Most will be stated in the pure second order case, but they provide impetus for investigating them at all levels of the potential theory hierarchy. Here we are sketchy, leaving many defintions to the references. Hence, rather than giving formal statements of theorems, we will recall the main results informally. 

\noindent {\bf 1) The Andreotti-Frankel Theorem for subsequations $\cF \subset \cS(n)$:}
(see \cite{HL11b} for details). First, {\em $\cF$-convex domains} can be defined. Then the notion of {\em $\cF$-free submanifolds} extends that of totally real submanifolds and one has the following result.
\begin{center}
	{\em An $\cF$-convex domain has the homotopy type of a $CW$-complex of dimension less than or equal to the maximal dimension of an $\cF$-free subspace.}
\end{center}
This maximal dimension is easy to calculate in the multitude of examples.

\noindent {\bf 2) The Levi Problem in $p$-convex geometry:} In  \cite{HL13c} this problem is solved, including the case of non-integer $p$. One has
\begin{center}
{\em Local $p$-convexity implies global $p$-convexity:}
\end{center}

\noindent {\bf 3) Boundary pseudoconvexity:} For bounded domains $\Omega$ with smooth boundary the following local to global result was established in Theorem 5.12 of \cite{HL11a} for {\em cone subequations $\cF \subset \cS(n)$}:
\begin{center}
	{\em If $\partial \Omega$ is strictly $\cF$-pseudoconvex at each point, then $\partial \Omega$ has a smooth strictly $\cF$-subharmonic defining function.}
\end{center}
For subequations $\cF \subset \cS(n)$ which are not cones, boundary convexity for $\cF$ is governed by the asymptotic behavior of $\cF$ at infinity. This asymptotic behavior is captured by a new subequation $\overrightarrow{\cF}$ which is a cone, so that the above can be applied.
See also Corollary 11.8 of \cite{HL11a} (and an incomplete discussion after the proof) for the general case. 

If strictness is dropped, then boundary pseudoconvexity does not imply existence of a global plurisubharmonic defining function in complex analysis. On the other hand, Forsterni\v{c}  recently proved that this is the case if $G = \G(p,\R^n)$ and $\cF = G^{\circ}$ (see \cite{F22}). This is particularly interesting as it runs counter to our point made here that generally speaking several complex variables is usually the source of results for the other potential theories.

\noindent {\bf 4) $\cF$-pluriharmonics for subequations $\cF \subset \cS(n)$:} These functions are the analogue of the real part of holomorphic functions . They are defined by requiring that the second derivative belongs to the largest linear subspace of $\cF$, refered to as the {\em edge} of $\cF$ (see \cite{HL19b}).  In particular, in Theorem 9.3 of \cite{HL19b} conditions on $\cF$ are found that ensure that the family of functions that can be written locally as the maximum of a finite number of $\cF$-pluriharmonics suffices for solving the $\cF$-Dirichlet problem via the Perron process. 

\noindent {\bf 5) Removable singularities for subequations $\cF$ on manifolds:} Pluripotential methods for proving removable singularity theorems in several complex variables can be extended to $\cF$-pluripotential theory and used to prove
\begin{center}
	{\em Removable singularity theorems in $\cF$-potential theory.}
\end{center} See \cite{HL14a} for details. 

\noindent {\bf 6) Tangents to subharmonics:} Kieselman's theory of tangents to plurisubharmonic functions in complex analysis can be extended to $\cF$-plurisubharmonic functions if $\cF$ is in a broad family of convex cones. The tangents to an $\cF$-plurisubharmonic function $u$ can be used to study the singularities of $u$.  See \cite{HL18a} and \cite{HL17a} for details.  It is shown that
\begin{center}
	{\em Tangents always exist and are maximal functions.}
\end{center} 
Maximal functions can be thought of as $\cF$-subharmonic functions with certain singularities allowed. The strongest form of uniqueness of tangents is when {\em strong uniqueness} holds; that is, the tangent of an $\cF$-plurisubharmonic function equals the density times the Riesz kernel. It is shown that, except for $\cP$ (where it is false)
\begin{center}
	{\em Strong uniqueness of tangents holds if $\cF$ is $\mathcal{O}(n)$ invariant.}
\end{center}
This strong uniqueness fails in all three of the basic cases $\cF = \cP$, $\cF = \cP_{\CF}$ (the case studied by Kieselman) and $\cF = \cP_{\HF}$ (the quarternionic case).

\noindent {\bf 7)  A Bombieri-H\"{o}rmander-Siu type structure theorem:} 
The result for the sets of high density for a plurisubharmonic function in complex analysis has a weakened version which extends to $\cF$-subharmonic functions for many convex $\cF \subset \cS(n)$, concluding that
\begin{center}
{\em Strong uniqueness of tangents implies that sets of high density are discrete.}
\end{center}
See section 14 of \cite{HL18a} for details. 

Recently, Chu \cite{Chu21} dramatically improved this result by showing that the singular set of an $\cF$-subharmonic function stratifies, and proving each stratum is a rectifiable set.

\subsection{The Dirichlet problem.}\label{sec:DP} 

Much can be said about the interplay between potential theory and operator theory by studying the Dirichlet problem. We will focus here on the Euclidian (coordinate chart) setting. We assume that $X$ is an open subest in $\R^n$ and consider bounded domains $\Omega \subset \subset X$ with smooth (i.e.\ $C^2$) boundaries and boundary data functions $\varphi \in C(\partial \Omega)$. One can state the standard Dirichlet problem in a vague form as:

\noindent {\bf (DP) - Vague Formulation:} Find a function  $h \in C(\overline{\Omega})$ which satisfies:

\begin{itemize}
	\item[1)] $h$ is a {\em solution} on $\Omega$, and
	\item[2)] $h_{|\partial \Omega} = \varphi$.
\end{itemize}

In order to be precise as to the meaning of solution in  1), one needs to start with either a subequation constraint set $\cF \subset \cJ^2(X)$ or an operator $F$ on the $2$-jets $\cJ^2(X)$. We want to make both choices and then bring them together. 

The Dirichlet problem involves three basic questions: {\bf {\em uniqueness}}, {\bf {\em existence}} and {\bf {\em regularity}}. Uniqueness follows from the {\em comparison principle}
\begin{equation}\label{CP}
\text(Comparison) \quad	\quad u \leq w \ \ \text{on} \ \partial \Omega \ \ \Rightarrow \ \ 	u \leq w \ \ \text{on} \ \Omega,
\end{equation}
for every pair of subharmonics/superharmonics for $\cF$ (or for every pair of subsoltions/supersolutions to the equation $F(J^2u) = 0$). Existence is established by {\em Perron's method}. The candidate solution is defined pointwise as the upper envelope
\begin{equation}\label{perron}
u(x) := \sup_{w \in \mathfrak{F}} w(x), \ \ x \in \Omega,
\end{equation}
of the Perron family $\mathfrak{F}$ of subsolutions $w$ with $w_{|\partial \Omega} \leq \varphi$. For comparison, roughly speaking, the size, but not the shape of the domain $\Omega$ can be of importance. By contrast, for existence one has a dichotomy between subequations $\cF$ where existence holds for all domains (with $\partial \Omega$ smooth), and subequations $\cF$ with an interesting distinguished boundary geometry of {\em $\cF$-pseudoconvexity} required for existence (see subsection \ref{sec:convexity} below). It is important to have a condition on the boundary $\partial \Omega$ which is a local (geometrical) requirement. Finally, the vast and important regularity question will, in essence, not be treated here. 

We now begin to describe the main ingredients in the two approaches, potential theory and operator theory. The potential theoretic formulation starts with a constraint set $\cF \subset \cJ^2(X)$, while the operator theoretic formulation starts with an operator $F$ whose domain is a subset $\cG \subset \cJ^2(X)$ ($\cG \equiv \cJ^2(X)$ is allowed and, in fact, is frequently required in the literature). Additional conditions must be imposed in either case. 

\noindent {\bf (DP) - Potential Theoretic Formulation:} Find $h \in C(\overline{\Omega})$ which satisfies:

\begin{itemize}
\item[1a)] $h$ is {\em $\cF$-subharmonic on $\Omega$} (i.e.\ $J^{2,+}_x h \subset \cF_x$ for each $x \in \Omega$). 
\item[1b)] $-h$ is {\em $\wt{\cF}$-subharmonic on $\Omega$} (i.e.\ $J^{2,+}_x(-h) \subset \wt{\cF}_x$ for each $x \in \Omega$). Equivalently, we will say that $h$ is {\em $\cF$-superharmonic on $\Omega$}.    
\item[2)] $h_{|\partial \Omega} = \varphi$.
\end{itemize}
Here and below, $\Int$ denotes the interior of a set, $\cF_x := \{ J \in \cJ^2: (x,J) \in \cF \}$ is the {\em fiber of $\cF$ over $x$}, and $\wt{\cF}$ is the {\em dual} of $\cF$ (see Definition \ref{defn:dual}). The space $J^{2,+}_x h$ of {\em upper test jets of $h$ at $x$} are defined in \eqref{UCJ} (also see Defintion \ref{defn:Fsub}). One can easily see that the definition 1b) is equivalent to saying that the lower test jets of $h$ satisfy $J^{2,-}_x h \subset ( \Int \, \cF_x)^{c}$, for each $x \in \Omega$. Additional conditions placed on the constraint set $\cF$ are made precise in Definition \ref{defn:subeq}. They are summarized by saying that $\cF$ is a {\bf {\em subequation}}.

The operator theoretic formulation of (DP), although standard, requires some explanation as well as some conditions on the operator $F$ and its domain $\cG$, which we give now for the sake of completeness. 

\begin{defn}[Proper elliptic operators]\label{defn:PEO} An operator $F \in C(\cG)$ where either
	\begin{equation*}\label{case1}
	\mbox{$\cG = \J^2(X)$ } \quad \text{({\bf {\em unconstrained case}})}
	\end{equation*}
	or
	\begin{equation*}\label{case2}
	\mbox{$\cG \subsetneq  \J^2(X)$ is a subequation constraint set  \quad \text{({{\bf \em constrained case}})}.}
	\end{equation*}
	is said to be {\em proper elliptic} if for each $x \in X$ and each $(r,p, A) \in \cG_x$ one has
	\begin{equation}\label{PEO}
	F(x,r,p,A) \leq F(x,r + s, p, A + P) \ \quad \forall \, s \leq 0 \ \text{in} \ \R \ \text{and} \  \forall \, P \geq 0 \ \text{in} \ \cS(n).
	\end{equation}
	The pair $(F, \cG)$ will be called a {\em proper elliptic \footnote{Such operators are often refered to as {\em proper} operators (starting from \cite{CIL92}). We prefer to maintain the term  ``elliptic'' to emphasise the importance of the {\em degenerate ellipticity} ($\cP$-monotonicity in $A$) in the theory.}
		 (operator-subequation) pair}.
\end{defn}

The minimal monotonicity \eqref{PEO} of the operator $F$ parallels the minimal monotonicity properties (P) and (N) for subequations $\cF$. It is needed for {\em coherence} and eliminates obvious counterexamples for comparison. This is explained for subequations after Definition \ref{defn:Fsub}. A given operator $F$ must often be restricted to a suitable background constraint domain $\cG \subset \cJ^2(X)$ in order to have this minimal monotonicity (the constrained case). The historical example clarifying the need for imposing a constraint is the Monge-Amp\`ere operator
\begin{equation}\label{MAE}
F(D^2u) = {\rm det}(D^2u),
\end{equation}
where one restricts the operator's domain to be the convexity subequation $\cG = \cP := \{ A \in \cS(n): A \geq 0 \}$. The scope of the constrained case is perhaps best illustrated by the more general {\em G\aa rding-Dirichlet operators}. See Example \ref{exe:GDP}, with \eqref{MAE} the fundamental case. These polynomial operators $F$ of degree $m$ have an ordered sequence of {\em G\aa rding eigenvalues} $\Lambda_1(A) \leq  \cdots \Lambda_m(A)$ which detemine {\bf {\em branches}} of the equation $F(J^2u)=0$. The notion of branches well illustrates the interplay between potential theory and operator theory and will be discussed in Example \ref{exe:branches}.

The unconstrained case, in which $F$ is proper elliptic on all of $\cJ^2(X)$ is the case usally treated in the literature and is perhaps best illustrated by the canonical operators mentioned above. 

We now recall the precise notion of solutions in the operator theoretic formulation of the Dirichlet problem. The defintions again make use of {\em upper/lower test jets}.

\begin{defn}[Admissible viscosity solutions]\label{defn:AVS} Given $F \in C(\cG)$ with $\cG \subset \cJ^2(X)$ a subequation on an open subset $X \subset \R^n$:
	\begin{itemize}
	\item[(a)] a function $u \in \USC(X)$ is said to be an {\em ($\cG$-admissible) viscosity subsolution of $F(J^2u)=0$ on $X$} if for every $x \in X$ one has
	\begin{equation}\label{AVSub}
	\mbox{$J \in J^{2, +}_{x}u \ \ \Rightarrow \ \   J \in \cG_x$ \ \ \text{and} \ \ $F(x, J) \geq 0$;}
	\end{equation}
	\item[(b)] a function $u \in \LSC(\Omega)$ is said to be an {\em ($\cG$-admissible) viscosity supersolution  of $F(J^2u)=0$ on $X$} if for every $x \in X$ one has
	\begin{equation}\label{AVSuper}
	\mbox{$J \in J^{2, -}_{x}u  \ \ \Rightarrow$ \ \ either [ $J \in \cG_x$ and \ $F(x,J) \leq 0$\, ] \ or \ $J \not\in \cG_x$.}
	\end{equation}
\end{itemize}
A function $u \in C(\Omega)$ is an {\em ($\cG$-admissible viscosity) solution of $F(J^2u)=0$ on $X$} if both (a) and (b) hold.
\end{defn}
In the unconstrained case where $\cG \equiv \cJ^2(X)$, the definitions are standard. In the constrained case where $\cG \subsetneq  \J^2(X)$, the definitions give a systematic way of doing of what is sometimes done in an ad-hoc way (see \cite{IL90} for operators of Monge-Amp\`{e}re type and \cite{Tr90} for prescribed curvature equations.) Note that \eqref{AVSub} says that the subsolution $u$ is also $\cG$-subharmonic and that \eqref{AVSuper} is equivalent to saying that $F(x,J) \leq 0$ for the lower test jets which lie in the constraint $\cG_x$.

\newpage

\noindent {\bf (DP)$'$ - Operator Theoretic Formulation:} Find $h \in C(\overline{\Omega})$ which satisfies:

\begin{itemize}
	\item[1a)$'$] $h$ is a {\em ($\cG$-admissible) subsolution of $F(J^2h)=0$ on $\Omega$} (Definition \ref{defn:AVS}(a)).
	\item[1b)$'$]  $h$ is a {\em ($\cG$-admissible) supersolution of $F(J^2h)=0$ on $\Omega$} (Definition \ref{defn:AVS}(b)).
	\item[2)] $h_{|\partial \Omega} = \varphi$.
\end{itemize}

We now discuss the equivalence of the potential theoretic and operator theoretic formulations of the Dirichlet problem; that is, the equivalence of (DP) for a given subequation $\cF$ and (DP)$'$ for a given (proper elliptic) opertor-subequation pair $(F, \cG)$. By the definitions, the equivalence of 1a) and 1a)$'$ is the same as the following equivalence: for each $x \in \Omega$ one has
$$
	J^{2,+}_x h \subset \cF_x \ \Longleftrightarrow \ \text{both} \quad J^{2,+}_x h \subset \cG_x \quad \text{and} \quad  F(x,J) \geq 0 \ \ \text{for each} \ \ J \in J^{2,+}_x h.
$$
This holds if and only if one has the {\bf{\em correspondence relation}}
\begin{equation}\label{relation}
\cF = \{ (x,J) \in \cG: \ F(x,J) \geq 0 \}.
\end{equation}
In addition, the equivalence of 1b) and 1b)$'$ is the same as the following equivalence: for each $x \in \Omega$ one has
\begin{equation}\label{equivalence}
J^{2,+}_x (-h) \subset \wt{\cF}_x \ \Longleftrightarrow \  \ J \not\in \cG_x \ \text{or} \ [J \in \cG_x \text{and} \ F(x,J) \leq 0], \ \forall \,  J \in J^{2,-}_x h.
\end{equation}
Using duality \eqref{dual_fiber} and $J^{2,+}_x (-h) = -J^{2,-}_x h$ one can see that that the equivalence \eqref{equivalence} holds if and only if one has {\bf {\em compatibility}}
\begin{equation}\label{compatibility1}
\Int \, \cF = \{ (x,J) \in \cG: \ F(x,J) > 0\}.
\end{equation} 
which for subequations $\cF$ defined by \eqref{relation} is equivalent to
\begin{equation}\label{compatibility2}
\partial  \cF = \{ (x,J) \in \cF: \ F(x,J) = 0\}.
\end{equation}
The pair of equivalences  $1a) \Leftrightarrow 1a)'$ and $1b) \Leftrightarrow 1b)'$ is refered to as the {\bf {\em correpsondence principle}} and will be discussed futher in section \ref{sec:correspndence}. 
These considerations can be summarized in the following result.

\begin{thm}[Correspondence Principle]\label{thm:corresp_gen}
	Suppose that $F \in C(\cG)$ is proper elliptic and $\cF$, defined by the correspondnce relation \eqref{relation}, is a subequation. If compatibility \eqref{compatibility1} is satisfied, then $h \in C(\overline{\Omega})$ satisfies the correspondence principle: $1a) \Leftrightarrow 1a)'$ and $1b) \Leftrightarrow 1b)'$. In conclusion, the two formulations (DP) and (DP)$'$ are equivalent. 
	\end{thm}

\begin{rem}[On compatibility]
Given one of either a proper elliptic pair $(F, \cG)$ or a subequation $\cF$, finding the other so that both the correspondence relation \eqref{relation} and compatibility \eqref{compatibility1} hold can be impossible, easy or in between requiring some work. For example, given any subequation $\cF$ the pair $(F, \cJ^2(X))$ with $F$ the signed distance operator \eqref{SDO} will do. Other natural choices of $(F, \cJ^2(X))$ with $\cF$ given, which require some additional work, are the canonical operators introduced in \cite{HL19a}. In the other direction, in subsection \ref{sec:PDE} we will present various examples of determining the subequation $\cF$ given a proper elliptic pair $(F, \cG)$. In particular, finding $\cF$ is easy in Examples \ref{exe:PMA} and \ref{exe:SLPE}, requires some work in Examples \ref{exe:k_Hessian}, \ref{exe:HASE}, \ref{exe:GDP} and \ref{exe:OTE}, and is impossible in Example \ref{exe:failure}.
\end{rem}

\subsection{Boundary pseudoconvexity}\label{sec:convexity} The potential theoretic approach to the Dirichlet problem (DP) naturally leads to an appropriate notion of pseudoconvexity for $\partial \Omega$ (smooth) required for existence. This is perhaps best illustrated by focusing on the case of a constant coefficient pure second order subequation $\cF \subset \cS(n)$ which is a cone. The definition, with roots in \cite{CNS85}, is given in \cite[section 5]{HL09c} with several equivalent formulations.  

\begin{defn}[Boundary pseudoconvexity]\label{defn:convex} A smooth boundary $\partial \Omega$ is said to be  {\em strictly $\cF$-pseudoconvex at $x \in \partial \Omega$} if 
	\begin{equation}\label{convexity}
	\exists \,  t_0 > 0 \quad \text{such that} \quad \ A_x + t P_{e(x)} \in \Int \, \cF, \quad  \forall \, t \geq t_0, 
	\end{equation}
	where $A_x$ denotes the second fundamental form of $\partial \Omega$ at $x$ with respect to the inward pointing unit normal $e(x)$ and $ P_{e(x)}$ is orthogonal projection onto the normal line through $e(x)$ (the eigenvalues of $A_x$ are the principal curvatures of $\partial \Omega$ at $x$).
	\end{defn}

These (cone) subequations $\cF$ divide into two kinds, those with and those without a boundary geometry. By those without a boundary  geometry we mean that all boundaries $\partial \Omega$ are strictly $\cF$-pseudoconvex at all points. This is equivalent to requiring that
\begin{equation}\label{convexity2}
\forall \, A \in \cS(n),  \forall \, e \in S^{n-1} \quad  \exists\ t_0 > 0 \quad \text{such that} \quad \ A_x +t  P_{e} \in \Int \, \cF, \ \ \ \forall \, t \geq t_0, 
\end{equation}
Taking $A = 0$ implies
\begin{equation}\label{convexity3}
 P_{e} \in \Int \, \cF \ \text{for every} \ e \in S^{n-1}.
\end{equation}
Conversely, if \eqref{convexity3} holds then $P_e + \veps A \in \Int \, \cF$ for $\veps > 0$ small, which is equivalent to \eqref{convexity2}. This proves that $\cF$ has no boundary (geometric) restriction for existence for the (DP) if and only if $\cF$ is strictly elliptic, since \eqref{convexity3} is one of the ways of defining strict ellipticity. This proves that
\begin{center}
{\em the strictly elliptic potential theories are exactly the ones without a boundary pseudoconvexity geometry.}
\end{center}
It is easy to see that none of the geometric potential theories $\cF = G^{\circ}$, $G$ a closed subset of $\mathbb{G}(n, \R^n)$), satisfy \eqref{convexity3}; i.e.\ none are strictly elliptic, so each has a boundary geometry. This geometry has the following nice description
\begin{center}
	{\em $\partial \Omega$ is strictly $\cF = G^{\circ}$-pseudoconvex at $x \in \partial \Omega \ \Leftrightarrow$ \ the restriction of the second fundamental form $A_x$ to any $k$-plane $W \in G$ which is also tangential, i.e.\ $W \subset T_x \partial \Omega$, has strictly positive trace.}
	\end{center}

\subsection{Some examples of PDEs.}\label{sec:PDE}  The potential theory appoach to treating nonlinear PDEs is well illustrated by many examples of operators (and classes of operators). We mention a few here.
	
\begin{exe}[{\bf Perturbed Monge-Amp\`ere}]\label{exe:PMA} With fixed $M \in C(\Omega, \cS(n))$ and $f \in C(\Omega)$ non-negative, consider
\begin{equation}\label{PMAE}
{\rm det}(D^2u + M(x)) = f(x), \ \ x \in \Omega \subset \subset \R^n 
\end{equation}
This is an important test example of Krylov \cite[Example 8.2.4]{Kv87} for probabilistic and analytic methods. It is also noteworthy because it fails to satisfy the standard viscosity structual conditions for comparison as given in Crandall-Ishii-Lions \cite[condition (3.14)]{CIL92} unless $M$ is the square of a Lipschitz continuous matrix valued function. In \cite{CP17}, comparison is proved for general continuous $M$ (along with the existence of a unique continuous solution of the Dirichlet problem on strictly convex domains). The potential theoretic proof, makes use of the compatible subequation whose fibers are defined by 
	$$
\cF_x := \{ A \in \cS(n): \ A + M(x) \geq 0 \ \text{and} \ F(x,A):= {\rm det}(A + M(x)) - f(x) \geq 0 \}.
	$$
This was done with the introduction and  application of the notion of (Hausdorff) continuity of the fiber map
	$$
\Theta: \Omega \to \wp(\cS(n)) \ \ \text{defined by} \ \Theta(x):= \cF_x, \ \ \forall \, x \in \Omega.
	$$
This is a representative example of the ``constrained case'' in which operators $F$ come with domains; that is, $F$ must be restricted to $\cG$ defined by its fibers 
	$\cG_x := \{ A \in \cS(n): \ A + M(x) \geq 0 \}$.
\end{exe}

\begin{exe}[{\bf Special Lagrangian potential equation}]\label{exe:SLPE} With {\em phase function} $\theta \in C(\Omega, I)$ where $I = (-n\pi/2, n\pi/2)$ consider
\begin{equation}\label{SLPE}
G(D^2u) := \sum_{k = 1}^n \arctan{(\lambda_k(D^2u))} = \theta(x), \ \ x \in \Omega \subset \subset \R^n 
\end{equation}
The geometric interpretation of this equation is that the graph of the gradient of $u$ will have Lagrangian phase $\theta$ (see \cite{HL82}). Comparison for constant phases, as well as existence/uniqueness for the Dirichlet problem via Perron's method, was proven in   \cite{HL09c}. For non-constant phases, comparison is difficult and not completely settled. The operator $G$ is particularly difficult to analyze if the non-constant phase $\theta$ assumes a {\em special phase value} $\theta_k = (n-2k)\pi/2, k = 1, \ldots n - 1$. Comparison was proven in \cite{DDT19} if $h$ has range in the first/last intervals $I_k$ determined by $\theta_k$. This is the ``relatively easy'' case where $G$ is concave/convex. The best result to date was obtained in \cite{CP21} for phases taking values in any phase interval
\begin{equation}\label{intervals}
I_k = (\theta_{k-1}, \theta_k), \ \ k = 1, \ldots n.
\end{equation}
There the key was to establish the fiber regularity of the fiber map 
	$$
\Theta(x):= \left\{ A \in \cS(n): \ \ F(x,A):= G(A) - h(x) \geq 0 \right\},
	$$ 
which is false across the special phase values. Combining comparison with the appropiate pseudoconvexity assumption on $\Omega$ yields existence/uniqueness for the Dirichlet problem for phases taking values in the intervals \eqref{intervals}, as shown in \cite{HL21a} (including a study of the needed pseudoconvexity).  This is a representative example of the unconstrained case (also pure second order) where the operator $G$ is increasing on all of $\cS(n)$.

\end{exe}

\begin{exe}[{\bf Eigenvalue equation for $k$-Hessian operators}]\label{exe:k_Hessian}  With $k = 1, \ldots n$ and $\mu \in \R$ fixed, consider 
\begin{equation}\label{EvalEq}
S_k(D^2u) + \mu u |u|^{k-1} = 0, \ \ x \in \Omega \subset \subset \R^n, 
\end{equation}
where for $A \in \cS(n)$ the {\em $k$-Hessian} operator is defined by
$$
S_k(A):=   \sigma_k(\boldsymbol{\lambda}(A)) =  \sigma_k(\lambda_1(A), \ldots, \lambda_n(A)) = \sum_{1 \leq i_1 < \cdots < i_k \leq n} \lambda_{i_1}(A) \cdots \lambda_{i_k}(A).
$$ 
Since the equation is $k$-homogeneous, one can search for eigen-directions (rays) $u$ that solve \eqref{EvalEq} for an eigenvalue $\mu$. The operator $S_k$ is degenerate elliptic (increasing in $A$) when restricted to the closed cone 
$$
\Sigma_k:=  \{ A \in \cS(A): \ \boldsymbol{\lambda}(A) \in \overline{\Gamma}_k \} \ \ \text{with} \ \Gamma_k := \{ \boldsymbol{\lambda} \in \R^n: \sigma_j(\boldsymbol{\lambda}) > 0, j = 1, \ldots, k \},
$$ 
the (closed) {\em G\aa rding cone} associated to $\sigma_k$ ($k$th elementary symmetric polynomial). Notice that $S_k$ interpolates between $S_1(D^2u)= {\rm tr}(D^2u) = \Delta u$ and $S_n(D^2u)= {\rm det}(D^2u)$. One uses  $\Sigma_k$ as a background (subequation cone) constaint set; that is, one looks for $k$-convex subharmonics and uses $k$-convex lower test functions for supersolutions. 
The interesting case concerns $u \leq 0$  and $\mu > 0$ where the equation has the  ``wrong''' monotonicity in $u$. In \cite{BP21}, a maximum principle characterization of a {\em generalized principal eigenvalue} in the sense of Berestycki-Nirenberg-Varadhan \cite{BNV94} is proven as well as the existence of a corresponding eigenfunction vanishing on the boundary. An important  step in the proof is to prove an a priori H\"{o}lder estimate, which is needed for compactness in an iterative scheme for the construction of the eigenfunction. The proof shows that the theory of admissibility constraints extends in a natural way the technique pioneered by Ishii-Lions \cite{IL90} in the unconstrained case.

\end{exe}

\begin{exe}[{\bf Hyperbolic affine sphere equation}]\label{exe:HASE} With $X \subset \R^n$ open and $f \in C(X)$ non-negative consider the following equation on $X$ 
\begin{equation}\label{HASE1}
[-u]^{n+2}{\rm det}(D^2u) = f.
\end{equation}
The geometric interpretation of the equation emerges by setting  $h := -f$ so that the equation becomes
\begin{equation}\label{HASE2}
{\rm det}(D^2u) = (h/u)^{n+2},
\end{equation}
which for $u$ convex and neqative describes the graphing function of a {\em hyperboic affine sphere with (constant) mean curvature $h \leq 0$} as discussed in Cheng-Yau \cite{CY86}. Comparison for the equation \eqref{HASE1} was established in  \cite{CP21}. This is another representative (gradient-free) example of the constrained case, where $(F, \cG)$ with 
$$
F(x,r,A):= (-r)^{n+2} {\rm det} \, A - f(x) \quad \text{and} \quad \cG = \cQ := \cN \times \cP. 
$$

\end{exe}

\begin{exe}[{\bf Optimal transport equations}]\label{exe:OTE} With $X \subset \R^n$ open, $f \in C(X)$ non-negative and $g \in C(\R^n)$ non-negative, consider the following equation on $X$
\begin{equation}\label{OTE}
g(Du) \, {\rm det}(D^2u) = f.
\end{equation}
The functions $f$ and $g$ represent the source and target densities respectively which should have the same mass ($L^1$-norm) (see \cite{DF14} and \cite{V} for more details.)

Comparison has been shown for constant target densities $g$ in \cite{CP21}. For non-constant $g$, one requires the additional monotonicity property of {\bf {\em directionality}}; that is, 
there exists a closed convex cone $\cD \subset \R^n$ with non-empty interior and vertex at the origin for which
\begin{equation}\label{directionality}
g(p + q) \geq g(p) \ \ \text{for each} \ \ p,q \in \cD.
\end{equation}
Examples of $g$ with directionality include
$$
g(p) = -p_n \ \ \text{with} \ \cD = \{ (p', p_n) \in \R^n: \ p_n \geq 0 \}
$$
and for $k \in \{ 1, \ldots n \}$
$$
g(p) = \prod_{j=1}^k p_j \ \ \text{with} \ \  D = \{ (p_1, \ldots p_n) \in \R^n:  p_j \geq 0 \ \text{for each} \ j = 1,\ldots k\} 
$$
For target densities $g$ with directionality, comparison has been shown for constant source densities $f$ in \cite{CHLP21} and for non-constant $f$ in \cite{CPR21}.
\end{exe}

\begin{rem}\label{rem:structure} Examples \ref{exe:HASE} and \ref{exe:OTE} have a product structure 
$$
F(x,u,Du,D^2u) = g(x,u)h(x,Du)G(x,D^2u) - f(x).
$$
This structure helps with the correspondence principle. One illustration of this is provided by considering the following pair of constant coefficient gradient-free operators
$$
	F(r,A) := -r \, {\rm det}(A) \quad \text{and} \quad G(r,A):= -r + {\rm det}(A).
	$$
The first operator is proper elliptic when restricted to the subequation $\cQ = \cN \times \cP$ and with $\cF:= \{ (r,A ) \in \cQ: \ \ F(r,A) \geq 0\}$ one has the compatibility
$$
		\partial \cF = \{ (r,A) \in \cF: \ F(r,A) = 0 \}
$$
and hence the correspondence principle. On the other hand, while $G$ is also proper elliptic when restricted to $\cQ$ (or even $\R \times \cP$), the boundary of
$$
	\cG := \{ (r,A) \in \cQ: \ G(r,A) \geq 0 \}
$$
includes $\cN \times \{0\}$, so that all negative $C^2$ affine functions will be $\cG$-harmonic but the operator $G$ is not zero on them. Thus the correspondence principle fails here. 
\end{rem}

Next we discuss perhaps what is perhaps the most interesting and important class  of examples. They illustrate why the constrained case is required.

\begin{exe}[{\bf G\aa rding-Dirichlet operators}]\label{exe:GDP} These nonlinear operators are obtained from G\aa rding's beautiful theory of hyperbolic polynomials \cite{Ga59}. We briefly give the precise definition and enumerate a few examples. See \cite{HL10} and \cite{HL09c}, \cite{HL11a}, \cite{HL13b} and \cite{CHLP21} for an extensive discussion.
	\end{exe}

\begin{defn}[Hyperbolic polynomials]\label{defn:HP} A homogeneous real polynomial $F$ of degree $m$ on $\cS(n)$ is {\em $I$-hyperbolic} if $F(I) > 0$ and for all $A \in \cS(n)$ the one variable polynomial $F(sI + A)$ has all $m$ roots real. 
\end{defn}

In keeping with the example $F(A) := {\rm det} \, A$, it is useful to focus on the negatives of the roots  of $F(sI + A)$, which are called {\em G\aa rding $I$-eigenvalues} and are denoted by $\Lambda_1(A), \ldots, \Lambda_m(A)$. Hence
\begin{equation}\label{GEvals}
	F(sI + A) = F(A) \prod_{j=1}^m (s + \Lambda_j(A)) \quad \text{and} \quad F(A) = F(I) \prod_{j=1}^m \Lambda_j(A).
	\end{equation}
The {\em open G\aa rding cone}
\begin{equation}\label{GCone}
	\Gamma:= \{ A \in \cS(n): \  \Lambda_j(A) > 0, j = 1, \ldots, m \}
\end{equation}
is a convex cone, which, along with $F$, has many nice properties. 

\begin{defn}[G\aa rding-Dirichlet operator\footnote{Perhaps they should be called {\em G\aa riding-Monge-Amp\`ere operators} instead.}]\label{defn:GDP} An $I$-hyperbolic polynomial operator $F$ of degree $m$ on $\cS(n)$ is called a  {\em G\aa rding-Dirichlet operator}  if $\cP \subset \overline{\Gamma}$; that is, if
\begin{equation}\label{GDP}	 
	A \geq 0 \ \ \Rightarrow \ \ \Lambda_j(A) \geq 0, \ \ j = 1, \ldots, m.
\end{equation}
In this case
\begin{equation}\label{GDPair}	
(F, \cF:= \overline{\Gamma}) \ \text{\em is a compatible operator-subequation pair}.
\end{equation}
\end{defn}
Also note that by \eqref{GDPair} $\Lambda_{\rm min}(A)$ is the canonical operator for the G\aa rding subequation $\cF:= \overline{\Gamma}$ since 
$$
		\Lambda_j(A + tI) = \Lambda_j(A) + t, \quad j = 1, \ldots, m.
$$

Here are some important examples. Let $\lambda_1(A), \ldots, \lambda_n(A)$ denote the standard eigenvalues of $A \in \cS(n)$. Of course, the Monge-Amp\`ere operator 
$$
		F(A) = {\rm det} A = \prod_{j=1}^n \lambda_j(A)
		$$
is the prototype. We leave it to the reader to see that these examples are polynomials. Perhaps the most interesting are those for which $\cF:= \overline{\Gamma}$ is geometrically defined; i.e. the cone subequation $\cF:= \overline{\Gamma}$ equals the polar $G^{\circ}$ of a closed subset $G$ of one the Grassmanians. Two new examples of this special nature are as follows.

\noindent {\bf 1.\ The $\mathbf{p}$-fold sum operator} (\cite[p.\ 39]{HL09c} and \cite[Proposition 7.11]{HL13a}): This is the operator
\begin{equation}\label{PFSO}
	F(A):= \prod_{i_1 < \cdots < i_p} (\lambda_{i_1}(A) + \cdots + \lambda_{i_p}(A)), \quad \text{where} \ p = 1, \ldots n.
\end{equation}
The degree of $F$ is $\binom{n}{p}$ and the closed G\aa rding cone is $\overline{\Gamma} = \G(p,\R^n)^{\circ}$. The canonical operator for $\overline{\Gamma}$ is the sum of the first (smallest) $p$ standard eigenvalues; i.e. the {\em $p$-th truncated Laplacian} (see \cite{HL09c} and \cite{BGI18}).

\noindent {\bf 2.\ The Lagrangian Monge-Amp\`ere operator} (\cite[section 5]{HL17b}): This is the operator
\begin{equation}\label{LMAO}
F(A):= \prod \left( \frac{1}{2} {\rm tr} A \pm \mu_1 \pm \cdots \pm \mu_n \right). 
\end{equation}
Here $A \in \cS(2n)$ is a real symmetric form on $\R^{2n} = \CF^n$, and $\pm \mu_1 \ldots \pm \mu_n$ are the eigenvalues of the skew Hermitian part of $A$. The $2^n$ G\aa rding eigenvalues are the factors $\frac{1}{2} {\rm tr} A \pm \mu_1 \pm \cdots \pm \mu_n$. The G\aa rding subequation $\overline{\Gamma}$ is geometric. It is the polar of $G:= {\rm LAG} \subset \G_{\R}(n, \CF^n)$, the set of Lagrangian $n$-planes. The plurisubharmonic functions; i.e.\ the $\overline{\Gamma}$-subharmonic functions are those upper semicontinuous functions that retrict to be $\Delta$-subharmonic on Lagrangian affine planes in $\CF^n$.

A classical example, which is not geometric for $k \neq 1, n$ is the $k$-Hessian operator $S_k(D^2u)$ discussed in Example \ref{exe:k_Hessian}.  Two more new non geometric examples, which have similarities with one another, are the following.

\noindent {\bf 3.\ The $\mathbf{\delta}$-uniformly elliptic operator} (\cite[Appendix B]{HL16c}): This is the operator 
\begin{equation}\label{DUEO}
F(A):= {\rm det} \left( A + \delta ({\rm tr} A)I \right) \ \text{with} \ \delta > 0,
\end{equation}
The G\aa rding eigenvalues of $F$ are
$$
	\Lambda_j(A) = \lambda_j(A) + \delta {\rm tr} (A), \quad j = 1, \ldots n.
$$
Hence the canonical operator for the G\aa rding subequation $\overline{\Gamma}$ is
$$
	\Lambda_{\rm min} (A) = \lambda_{\rm min} (A) + \delta \, {\rm tr} A.
$$

\noindent {\bf 4.\ The Pucci-G\aa rding-Monge-Amp\`ere  operator} (\cite[section 4.5]{HL13b}, \cite[Appendix B]{HL16c}, and in particular \cite[Example 6.10]{HL19a}): Fix $0 < \lambda < \Lambda$ and consider the the ``cube'' in eigenvalue space 
$$
	\C_{\lambda, \Lambda} := \{ A \in \cS(n): \ \lambda I \leq A \leq \Lambda I \} \subset \cP \subset \cS(n).
$$
Its polar $\cP^{-}_{\lambda, \Lambda} := \C_{\lambda, \Lambda}^{\circ}$ contains $\cP^{\circ} = \cP$ and hence is a subequation called the {\em Pucci} or {\em Pucci-G\aa rding cone}. The cone on $\C_{\lambda, \Lambda}$, denoted by ${\rm Cone}(\C_{\lambda, \Lambda})$, has a finite set of extreme rays through a subset $S$ of the vertices of $\C_{\lambda, \Lambda}$. The {\em Pucci-G\aa rding-Monge-Amp\`ere  operator} $F_{\lambda, \Lambda}$ is defined to be the product of the linear functionals in the set $S$. The factors are the G\aa rding $I$-eigenvalues of $F_{\lambda, \Lambda}$. The closed G\aa rding cone is the Pucci cone $\cP^{-}_{\lambda, \Lambda}$. Its canonical operator is the minimal eigenvalue, which is easily seen to be
$$
	\Lambda_{\rm min}(A) = \lambda {\rm tr} \, A^+ + \Lambda {\rm tr} \, A^- \geq 0, 
$$
where $A = A^+ + A^-$ is the decomposition of $A$ into positive and negative parts, and hence 
$$
\cP_{\lambda, \Lambda}^-:= \{A \in \cS(n): \ \lambda {\rm tr} \, A^+ + \Lambda {\rm tr} \, A^- \geq 0 \}.
$$
Note that the Pucci-G\aa rding operator $F_{\lambda, \Lambda}$ has degree $|S|$. 

The operator $S_k$ and the operators 1., 3., and 4.\ above which involve the real eigenvalues of $A \in \cS(n)$ have complex and quaternionic analogues that are also G\aa rding-Dirichlet operators. See \cite[section 5]{HL10}, \cite[section 10]{HL09c} and \cite[section 15]{HL11a} for more details.

The next family of examples provides a good illustration of the interplay between potential  theory and operator theory. 

\begin{exe}[{\bf Branches}]\label{exe:branches}
	The potential theory/subequation approach provides a direct way of extending the Dirichlet problem (DP) for the Monge-Amp\`ere operator to the other {\em branches} $\mathbf{\Lambda}_k$ of ${\rm det} \, (D^2u) = 0$, where, except for $k=1$, there is no natural smooth operator $F$ defining the solutions (or $\mathbf{\Lambda}_k$-subharmonics). The branch $\mathbf{\Lambda}_k \subset \cS(n)$ for $k = 1, \ldots n$ is the subequation defined by
	\begin{equation}\label{MA_branch}
	\mathbf{\Lambda}_k := \{ A \in \cS(n): \ \lambda_k(A) \geq 0 \},
	\end{equation}
	where $\lambda_1(A) \leq \cdots \leq \lambda_n(A)$ are the ordered eigenvalues of $A \in \cS(n)$. Despite the fact that ${\rm det} \, (D^2u)$ is not a proper elliptic operator on $\mathbf{\Lambda}_k$, the (DP) for $\mathbf{\Lambda}_k$-harmonics is meaningful. Existence and uniqueness for all boundary functions $\varphi \in C(\partial \Omega)$ under the appropriate geometrical conditions on $\partial \Omega$ was established in \cite{HL09c}. This theorem extends to branches of a more general G\aa rding-Dirichlet operator $F = \mathfrak{g}$ of degree $m$ defined by
	$$
	\mathbf{\Lambda}_k^{\mathfrak{g}} := \{ A \in 	\cS(n): \Lambda^{\mathfrak{g}}_k(A) \geq 0 \}, \ \ k = 1, \ldots m
	$$
. 
	
\end{exe}

These branches have a so-called {\bf {\em canonical operator}}, which we discuss next. Canonical operators well illustrate the unconstrained case.

\begin{exe}[{\bf Canonical operators}]\label{exe:COs} For clarity we focus on the pure second order constant coefficient subequation case $\cF \subset \cS(n)$. See the list of refernces before the formula \eqref{SDO} above for more information. A {\em canonical operator} for $\cF$ is by definition a function $F \in C(\cS(n))$ with the following two properties:
	\begin{equation}\label{CO1}
	F(A + P) \geq F(A), \ \ \forall \, A \in \cS(n) \ \text{and} \ P \geq 0 
	\end{equation}
and for some constant $c > 0$
	\begin{equation}\label{CO2}	
	F(A + tI) = F(A) + ct, \ \ \forall \, A \in \cS(n) \ \text{and} \ t \in \R. 
	\end{equation}
	
\begin{prop}[Existence and uniqueness of canonical operators]\label{prop:COs} Given a subequation $\cF \subset \cS(n)$, there exists a unique canonical operator $F$ with $\cF = \{A \in \cS(n): F(A) \geq 0 \}$.
\end{prop}

The proof can be summarized succinctly by defining $F$ by requiring
$$
	A_0 + F(A_0) I \in \partial \cF \ \ \text{for all} \ A_0 \perp I \ \text{(i.e.} \ {\rm tr} A_0 = 0)
$$
and then extending $F$ to all $A = A_0 + tI \in \cS(n), A_0 \perp I$ by formula \eqref{CO2}.

The canonical operator $F$ for some of the examples above are as follows. One has $\lambda_1(A)$ for the convexity subeqution $\cP$, $\lambda_1(A) + \cdots \lambda_p(A)$ for the $p$-fold subequation defined by \eqref{G_constraint} with $G = G(p, \R^n)$, $\frac{1}{2} {\rm tr} A - \mu_1 - \cdots - \mu_n$ for the Lagrangian subequation defined by \eqref{G_constraint} with $G = {\rm LAG}$. As mentioned above, the G\aa rding subequation $\cF = \overline{\Gamma}$ has canonical operator $\Lambda_{\rm min}(A)$, the minimal G\aa rding eigenvale operator. The $k$-th branch has canonical operator the $k$-th G\aa rding eigenvalue operator. The construction of a canonical operator extends to subequations $\cF \subset \cJ^2(X)$ if there is sufficent monotonicity (see section 11.4 of \cite{CHLP21}).
	\end{exe}

We conclude this subsetion with one last example. As we have indicated, a general principle is that comparison holds with sufficient monotonicity. With insufficient monotonicity comparison can fail even in the constant coefficient case and even on arbitrarily small balls.

\begin{exe}[{\bf Comparison fails}]\label{exe:failure} Consider the operators $F,G \in C(\R^n \times \cS(n), \R)$ defined by
$$
F(p,A):= \lambda_{\rm min}(M(p,A)) \quad \text{and} \quad G(p,A):= \lambda_{\rm max}(M(p,A))
$$ 	
with $M = M_{\alpha}$ is the $\cS(n)$-valued function defined for $\alpha \in (1, + \infty)$ by
$$
M(p,A):= A + |p|^{\frac{\alpha - 1}{n}} (P_{p^\perp} + \alpha P_p)) \quad \text{if} \ p \neq 0 \quad \text{and} \quad M(0,A):= A.
$$	
where for $p \neq 0$, $P_p, P_{p^\perp}$ are the projections onto the subspaces $[p], [p]^{\perp}$; that is,
\vspace{-0,5pc}
$$
P_p = \frac{1}{|p|^2} p \otimes p \quad \text{and} \quad P_{p^{\perp}} = I - P_p.
$$ 
These operators are studied in \cite{CHLP21}. Existence for the Dirichlet problem holds on all balls (with continuous Dirichlet data). The comparison principle, maximum principle and uniqueness of solutions fail on all arbitrarily small balls about every point. A partial explanation of these failures is that the {\em maximal monotonicity cones} for the associated (compatible) subequations $\cF,\cG$ are $\cM:= \{0\} \times \cP \subset \R^n \times \cS(n)$, which have empty interior.

\end{exe}

\section{Fundamental aspects of nonlinear potential theory}

In this section, we give a brief review of some key notions and fundamental results in the theory of $\cF$-subharmonic functions defined by a subeqation constraint set $\cF$.

\subsection{Subequations, subharmonics and duality}\label{sec:sub}

Suppose that $X$ is an open subset of $\R^n$ with $2$-jet space denoted by  $\cJ^2(X) = X \times (\R \times \R^n \times \cS(n))$. A good definition of a constraint set with a robust potential theory was given in \cite{HL11a} (also for manifolds).
	\begin{defn}[Subequations]\label{defn:subeq} A set $\cF \subset \cJ^2(X)$ is called a {\em subequation (constraint set)} if
		\begin{itemize}
			\item[(P)] $\cF$ satisfies the {\em positivity condition}
			(fiberwise); that is, for each $x \in X$
			$$
			(r,p,A) \in \cF_x \ \ \Rightarrow \ \ (r,p,A + P) \in \cF_x, \ \ \forall \, P \geq 0 \ \text{in} \ \cS(n).
			$$
			\item[(T)]  $\cF$ satisfies three conditions of {\em topological stability}: 
			\begin{equation}\tag{T1} 
		\cF = \overline{\Int \, \cF};
		\end{equation}
		\begin{equation}\tag{T2} 
		 \cF_x = \overline{\Int \, \left( \cF_x \right)}, \ \ \forall \, x \in X;
		\end{equation}
			\begin{equation}\tag{T3} \left( \Int \, \cF \right)_x = \Int \, \left( \cF_x \right), \ \ \forall \, x \in X.
			\end{equation}
			\item[(N)]  $\cF$ satisfies the {\em negativity condition}
			(fiberwise); that is, for each $x \in X$
			$$
			(r,p,A) \in \cF_x \ \ \Rightarrow \ \ (r + s,p,A) \in \cF_x, \ \ \forall \, s \leq 0 \ \text{in} \ \R.
			$$
		\end{itemize}
	\end{defn}
Notice that  by property (T1), $\cF$ is closed in $\cJ^2(X)$ and each fiber $\cF_x$ is closed in $\cJ^2$ by (T2). In addition, the interesting case is when each fiber $\cF_x$ is not all of $\cJ^2$, which we almost always assume. Also notice that in the constant coefficient pure second order case $\cF \subset \cS(n)$, property (N) is automatic and property (T) reduces to (T1)  $	\cF = \overline{\Int \, \cF}$, which is implied by (P) for $\cF$ closed. Hence in this case subequations $\cF \subset \cS(n)$ are closed sets simply satisfying (P).

 The conditions (P), (T) and (N) have various (important) implications for the potential theory determined by $\cF$. Some of these will be mentioned below (see the brief discussion following Definition \ref{defn:Fsub}).
 
 Next is {\em duality}, a notion first introduced in the pure second order coefficient case in \cite{HL09c}.
 
\begin{defn}[The dual subequation]\label{defn:dual} For a given subequation $\cF \subset \cJ^2(X)$ the 
{\em Dirichlet dual} of $\cF$ is the set $\wt{\cF} \subset \cJ^2(X)$ given by \footnote{Here and below, $c$ denotes the set theoretic complement of a subset.}
\begin{equation}\label{dual}
\wt{\cF} :=  (- \Int \, \cF)^c = - (\Int \, \cF)^c \ \ \text{(relative to \ 
$\cJ^2(X)$)}.
\end{equation}
 \end{defn}
With the help of property (T), the dual can be calculated fiberwise
\begin{equation}\label{dual_fiber}
\wt{\cF}_x :=  (- \Int \, (\cF_x))^c = - (\Int \, (\cF_x))^c \ \ \text{(relative to \  $\cJ^2$)}, \ \ \forall \, x \in X.
\end{equation}
This is a true duality in the sense that one can show
\begin{equation}\label{true_dual}
\wt{\wt{\cF}} = \cF \quad \text{and} \quad \text{$\cF$ is a subequation \ \ $\Rightarrow$ \ \  $\wt{\cF}$ is a subequation}. 
\end{equation}

Now comes the notion of {\em $\cF$-subharmonicity} for a given subequation $\cF \subset \cJ^2(X)$. There are two different natural formulations for differing degrees of regularity.
The first is the classical formulation.

\begin{defn}[Classical or $C^2$ subharmonics]\label{defn:CSH} A function $u \in C^2(X)$ is said to be {\em $\cF$-subharmononic on $X$} if
	\begin{equation}\label{VsubClass}
 J^2_x u := (u(x), Du(x), D^2u(x)) \in \cF_x, \ \ \forall \, x \in X
\end{equation}
with the accompanying notion of being {\em strictly $\cF$-subharmononic} if 
\begin{equation}\label{VsubCS}
J^2_x u \in \Int \, (\cF_x) = (\Int \, \cF)_x, \forall \, x \in X.
\end{equation}
 \end{defn}
 
 For merely upper semicontinuous functions $u \in \USC(X)$ with values in $[-\infty, + \infty)$, one replaces the $2$-jet $J^2_x u$ with the set of {\em $C^2$ upper test jets}
 	\begin{equation}\label{UCJ}
 J^{2,+}_{x} u := \{ J^2_{x} \varphi:  \varphi \ \text{is} \ C^2 \ \text{near} \ x, \  u \leq \varphi \ \text{near} \  x \ \text{with equality at} \ x \},
 	\end{equation}
 thus arriving at the following viscosity formulation.
 
 \begin{defn}[Semicontinuous subharmonics]\label{defn:Fsub}  A function $u \in \USC(X)$ is said to be {\em $\cF$-subharmonic on $X$} if
	\begin{equation}\label{Vsub}
	J^{2,+}_x u \subset \cF_x, \ \ \forall \, x \in X.
	\end{equation}
	We denote by $\cF(X)$ the set of all $\cF$-subharmonics on $X$.
\end{defn} 

We now recall some of the implications that properties (P), (T) and (N) have on an $\cF$-potential theory. Property (P) is crucial for {\bf {\em $C^2$-coherence}}, meaning classical $\cF$-subharmonics are $\cF$-subharmonics in the sense \eqref{Vsub}, since for $u$ which is $C^2$ near $x$, one has
$$
	J^{2,+}_x = J_x^2u + (0,0,\cP) \ \ \text{where} \ \ \cP = \{ P \in \cS(n): \ P \geq 0 \}.
$$
 The natural notion of $w \in \LSC(X)$ being {\em $\cF$-superharmonic} using {\em lower test jets} is  
	\begin{equation}\label{Vsuper1}
J^{2,-}_x w \subset \left(\Int \, (\cF_x)\right)^c, \ \ \forall \, x \in X,
	\end{equation}
	which by duality and property (T) is equivalent to $-w \in \USC(X)$ satisfying
	\begin{equation}\label{Vsuper2}
J^{2,+}_x (-w) \subset \wt{\cF}_x, \ \ \forall \, x \in X.
	\end{equation}
That is,  
	\begin{equation}\label{Vsuper3}
\text{$w$ is $\cF$-superharmonic \ \ $\Leftrightarrow$ \ \ 	$-w$ is $\wt{\cF}$-subharmonic}.
\end{equation}

Next note that property (T) insures the local existence of strict classical $\cF$-superharmonics at points $x \in X$ for which $\cF_x$ is non-empty. One simply takes the quadratic polynomial whose $2$-jet at $x$ is $J \in \Int \, (\cF_x)$. Finally, property
(N) eliminates obvious counterexamples to comparison. The simplest counterexample is provided by the constraint set $\cF \subset J^2(\R)$ in dimension one associated to the equation $u^{\prime \prime} - u = 0$.

\subsection{Monotonicity}

This fundamental notion appears in various guises. It is a useful and unifying concept. One says that a subequation $\cF$ is {\em $\cM$-monotone} for some subset $\cM \subset \J^2(X)$ if 
\begin{equation}\label{monotone}
\cF_x + \cM_x \subset \cF_x \ \ \text{for each} \  x \in X.
\end{equation}
For simplicity, we will restrict attention to (constant coefficient) {\em monotonicity cones}; that is, monotonicity sets $\cM$ for $\cF$ which have constant fibers which are closed cones with vertex at the origin. 

First and foremost, the properties (P) and (N) are monotonicity properties. Property (P) for subequations $\cF$ corresponds to {\em degenerate elliptic} operators $F$ and properties (P) and (N) together correspond to {\em proper elliptic} operators. Note that (P) plus (N) can be expressed as the single monotonicity property 
\begin{equation}\label{MM}
 \cF_x + \cM_0 \subset \cF_x   \ \ \text{for each} \ x \in X
\end{equation}
where
\begin{equation}\label{MMC}
\cM_0 := \cN \times \{0\} \times \cP \subset \cJ^2 = \R \times \R^n \times \cS(n)
\end{equation}
with
\begin{equation}\label{NP}
\cN := \{ r \in \R : \ r \leq 0 \} \quad \text{and} \quad  \cP := \{ P \in \cS(n) : \ P \geq 0 \} .
\end{equation}
Hence $\cM_0$ will be referred to as the {\em minimal monotonicity cone} in $\cJ^2$. However,  it is important to remember that $\cM_0 \subset \cJ^2$ is {\bf not} a subequation since it has empty interior so that property (T) fails. 

Combined with duality and fiberegularity (defined in subsection \ref{sec:FR}), one has a very general, flexible and elegant geometrical approach to comparison when a subsequation $\cF$ admits a constant coefficient monotonicity cone subequation $\cM$. A key ingredient to this approach is the  {\em Subarmonic Addition Theorem}:
\begin{equation}\label{SAT}
	\cF + \cM\subset \cF \ \ \Rightarrow \ \ \cF(X) + \wt{\cF}(X) \subset \wt{\cM}(X).
\end{equation}
This result reduces the comparison principle for $\cF$ to the  {\em Zero Maximum Principle} for the constant coefficient dual cone subequation $\wt{\cM}$; that is, for all $\Omega \subset \subset X$
\begin{equation}\label{ZMP}
\text{(ZMP)} \quad \quad	  \ z \leq 0 \ \text{on} \ \partial \Omega \ \ \Rightarrow z \leq 0 \ \text{on} \ \partial \Omega,
\end{equation}
$\forall \, z \in \USC(\overline{\Omega}) \cap \wt{\cM}(\Omega)$. This reduction of comparison to (ZMP) will be referred to as the {\bf {\em monotonicity-duality method}} and will be discussed in section \ref{sec:comparison}.

Monotonicity is also used to formulate {\bf {\em reductions}} when certain jet variables are ``silent'' in the subequation constraint $ \cF$. for example, one has
	$$
	\text{(pure second order)} \ \ \ \cF_x + \cM(\cP) \subset \cF_x: \ \ \cM(\cP):= \R \times \R^n \times \cP
	$$
	$$
	\text{(gradient free)} \ \ \ \cF_x + \cM(\cN,\cP) \subset \cF_x: \ \ \cM(\cN,\cP):= \cN \times \R^n \times \cP
	$$
	$\cM(\cP)$ and $\cM(\cN, \cP)$ are fundamental {\em constant coefficient (cone) subequations} which can be identified
	with $\cP \subset \cS(n)$ and $\cQ := \cN \times \cP \subset \R \times \cS(n)$. One can identify $\cF$ with subsets of the {\em reduced jet bundles} $X \times \cS(n)$ and $X \times (\R \times \cS(n))$, respectively, ``forgetting about'' the silent jet variables (see Chapter 10 of \cite{CHLP21}). 
	
	Two important ``reduced'' examples are worth drawing special attention to.

\begin{exe}[The convexity subequation]\label{exe:CSE} The {\em convexity subequation} is $\cF = X \times \cM(\cP)$ and reduces to $X \times \cP$ which has constant coefficients (each fiber is $\cP$) and for $u \in \USC(X)$ 
$$
 u \ \text{is $\cP$-subharmonic} \ \ \Leftrightarrow \ \ u \ \text{is locally convex} 
$$
(away from any connected compontents where $u \equiv - \infty$).

The convexity subequation has its {\em canonical operator} $F  \in C(\cS(n), \R)$ defined by the minimal eigenvalue $F(A) := \lambda_{\rm min}(A)$, for which 
\begin{equation}\label{PDual}
\cP = \{ A \in \cS(n): \ \lambda_{\rm min}(A) \geq 0 \}.
\end{equation} 
The dual subequation $\wt{\cF}$ has constant fibers given by 
\begin{equation}\label{Pdual}
\wt{\cP} = \{ A \in \cS(n) : \lambda_{\rm max}(A) \geq 0 \}
\end{equation}
which is the {\em subaffine subequation}. The set $\wt{\cP}(X)$ of dual subharmonics agrees with $\SA(X)$ the set of {\em subaffine functions} defined as those functions $u \in \USC(X)$ which satisfy the {\em subaffine property} (comparison with affine functions): for every $ \Omega \subset \subset X$ one has
	\begin{equation}\label{SAP}
u \leq a \ \ \text{on} \ \partial \Omega \ \ \Rightarrow \ \ u \leq a \ \ \text{on} \ \Omega, \ \ \text{for every $a$ affine}. 
	\end{equation}
The fact that $\wt{\cP}(X) = \SA(X)$ is shown in \cite{HL09c}. The subaffine property for $u$ is stronger than the maximum principle for $u$ since constants are affine functions. It has the advantage over the maximum principle of being a local condition on $u$. This leads to the comparison principle for all pure second order constant coefficient subequations \cite{HL09c} and extends to variable ccoefficient subequations \cite{CP21} using a notion of fiberegularity, as will be discussed in section \ref{sec:comparison}.

\end{exe}

\begin{exe}[The convexity-negativity subequation]\label{exe:CNSE} The constant coefficient gradient-free subequation $\cF= X \times \cM(\cN, \cP)$ reduces to $X \times \cQ \subset X \times (\R \times \cS(n))$ 
whose (constant) fibers are
\begin{equation}\label{Q}
\cQ = \cN \times \cP = \{ (r,A) \in \R \times \cS(n): r \leq 0 \ \ \text{and} \ \  A \geq 0 \}.
\end{equation}
The (reduced) dual subequation has (constant) fibers
\begin{equation}\label{Qdual}
\wt{\cQ} = \{ (r,A) \in \R \times \cS(n): r \leq 0 \ \ \text{or} \ \  A \in \wt{\cP} \}.
\end{equation}
The set $\wt{\cQ}(X)$ of dual subharmonics agrees with  $\SA^+(X)$, the set of {\em subaffine plus functions} defined as those functions $u \in \USC(X)$  which satisfy the {\em subaffine plus property}: for every $\Omega \subset \subset X$ one has
	\begin{equation}\label{SAPP}
	u \leq a \ \ \text{on} \ \partial \Omega \ \ \Rightarrow \ \ u \leq a \ \ \text{on} \ \Omega, \ \ \text{for every $a$ affine with} \ a_{|\overline{\Omega}} \geq 0, 
\end{equation}
from which the (ZMP) for $\wt{\cQ}$ subharmoncis follows immediately. 
The fact that $\wt{\cQ}(X) = \SA^+(X)$ is shown in \cite{CHLP21} together with the additional equivalence 
\begin{equation}\label{SAPP2}
\SA^+(X):= \{ u \in \USC(X): \ u^+ := \max \{u, 0 \} \in \SA(X) = \wt{P}(X) \},
\end{equation}
The validity of the (ZMP) for $\cQ$-subharmonics leads to the comparison principle by the monotonicity-duality method for all gradient free subequations with constant coefficients in \cite{CHLP21} and extends to variable coefficient gradient-free subequations in \cite{CP21}, using the notion of fiberegulaity.
\end{exe}

\subsection{Fiberegularity}\label{sec:FR} This fundamental notion can be used to pass from constant coefficient subequations (and operators) to ones with variable coefficients.

\begin{defn}\label{defn:fibereg} A subequation $\cF \subset \cJ^2(X)$ is {\em fiberegular} if the fiber map $\Theta$ is {\em (Hausdorff) continuous}; that is, if the set-valued map
	$$
	\Theta: X \to \cK(\cJ^2) \ \ \text{defined by} \ \  \Theta(x):= \cF_x,\ \ x \in X
	$$
	is continuous when the closed subsets $\cK(\cJ^2)$ of $\cJ^2$ are equipped with the {\em Hausdorff metric}
	$$
d_{\cH}(\Phi, \Psi) := {\rm max} \left\{ \sup_{J \in \Phi} \inf_{J' \in \Psi} || J - J'||, \sup_{J' \in \Psi} \inf_{J \in \Phi} || J - J'|| \right\}
	$$
	where 
	$$ 
	||J|| = ||(r,p,A)|| := \max \left\{ |r|, |p|, \max_{1 \leq k \leq n} |\lambda_k(A)| \right\}
	$$ 
	is taken as the norm on $\cJ^2$.
\end{defn}
This notion was first introduced in \cite{CP17} in the special case $\cF \subset X \times \cS(n)$. We will also refer to $\Theta$ as a {\em continuous proper ellipitc map} since it takes values in the closed (non-empty and proper) subsets of $\cJ^2$ satisfying properties (P) and (N). If $\cF$ is $\cM$-monotone for some (constant coefficient) monotonicity cone subequation, fiberegularity has the more useful equivalent formulation: there exists $J_0 \in \Int \, \cM$ such that for each fixed $\Omega \subset \subset X$ and $\eta > 0$ there exists $\delta = \delta(\eta, \Omega)$ such that
\begin{equation}\label{FC}
x,y \in \Omega, |x-y| < \delta \ \ \Longrightarrow \ \  \Theta(x) + \eta J_0 \subset \Theta(y).
\end{equation}
Note that property holds for each fixed $J_0 \in \Int \, \cM$ (see \cite{CPR21}) and that in the pure second order and gradient-free cases there is a ``canonical'' reduced jet $J_0 = I \in \cS(n)$ and $J_0 = (-1,I) \in \R \times \cS(n)$, respectively.  
Also note that fiberegularity is uniform on bounded domains as the $\delta$ in \eqref{FC} is independent of $x,y \in \Omega$.

Fiberegularity is crucial since it implies the {\bf {\em uniform translation property}} for subharmonics : {\em if $u \in \cF(\Omega)$, then there are small $C^2$ strictly $\cF$-subharmonic perturbations of \underline{all small translates} of  $u$ which belong to  $\cF(\Omega_{\delta})$}, where  $\Omega_{\delta}:= \{ x \in \Omega: d(x, \partial \Omega) > \delta \}$.

\begin{thm}[Uniform translation property for subharmonics]\label{thm:UTP}
	Suppose that a subequation $\cF$ is a fiberegular and $\cM$-monotone on $\Omega \subset \subset \R^n$ for some monotonicity cone subequation $\cM$. Suppose that $\cM$ admits a strict approximator; that is, there exists $\psi \in \USC(\overline{\Omega}) \cap C^2(\Omega)$ which is strictly  $\cM$-subharmonic   on  $\Omega$. Given $u \in \cF(\Omega)$, for each $\theta > 0 $ there exist $\eta = \eta(\psi, \theta) > 0$ and $\delta = \delta(\psi, \theta) > 0$ such that
	\begin{equation}\label{UTP}
 u_{y,\theta} = \tau_yu + \theta \psi \ \ \text{belongs to} \ \cF(\Omega_{\delta}), \ \ \forall \, y \in B_{\delta}(0),
	\end{equation}
	where $\tau_y u(\, \cdot \, ):= u(\, \cdot - y)$.
\end{thm}

In the pure second order and gradient-free cases ($\cF \subset \Omega \times \cS(n)$ and $\cF \subset \Omega \times (\R \times \cS(n))$, one always has a quadratic strict approximator $\psi$ and the theorem holds for all continuous coefficient $\cF$ which are minimally monotone (with  $\cM = \cP \subset \cS(n)$ and $\cM = \cQ = \cN \times \cP \subset \R \times \cS(n)$ respectively) as shown in \cite{CP17}, \cite{CP21}. The general $\cM$-monotone and fiberegular case is treated in \cite{CPR21}. In this general case, the hypothesis of the existence of a stict approximator $\psi$ creates no additional problems if the objective is to prove comparison. This is because one knows from \cite[Theorem 6.2]{CHLP21} that the existence of a strict approximator $\psi$ for $\cM$ ensures the validity of the (ZMP) for $\wt{\cM}$, which is needed for our monotonicity duality method. Moreover, the (constant coefficient) monotonicity cone subequations which admit strict approxiamtors are well understood by the study made in \cite{CHLP21}.

\subsection{Subharmonic addition for quasiconvex functions} Many results about $\cF$-subharmonic functions $u$ are more easily proved if one assumes that $u$ is also locally quasiconvex\footnote{We have adopted the term quasiconvex which is consistent with the use of {\em quasi-plurisubharmonic} function in several complex variables. Quasiconvex functions are sometimes referred to as {\em semiconvex} functions, although this term is a bit misleading. They are functions whose Hessian (in the viscosity sense) is locally bounded from below.}. Then, one can make use of {\em quasiconvex approximation} to extend the result to semicontinuous $u$. Here we discuss some of the main results in this direction. See \cite{PR22} for an extensive treatment, which borrows heavily from \cite{HL16a} and \cite{HL16b}.

\begin{defn}
	A function $u: C \to \R$ is {\em $\lambda$-quasiconvex} on a convex set  $C \subset \R^n$ if there exists $\lambda \in \R_{+}$ such that $u + \frac{\lambda}{2} | \, \cdot \, |^2$ is convex on $C$. A function $u: X \to \R$ is {\em locally quasiconvex} on an open set  $X \subset \R^n$ if for every $x \in X$, $u$ is $\lambda$-quasiconvex on some ball about $x$ for some $\lambda \in \R_{+}$.  
\end{defn}

Such functions are twice differentiable for almost every\footnote{With respect to the Lebesgue measure on $\R^n$.}  $x \in X$ by a very easy generalization of Alexandroff's theorem for convex functions (the addition of a smooth function has no effect on differentiability). This is one of the many properties that quasiconvex functions inherit from convex functions. Quasiconvex functions are used to approximate $u \in \USC(X)$ (bounded from above) by way of the {\em sup-convolution}, which for each $\veps > 0$ is defined by
\begin{equation}\label{sup_conv}
	 u^{\veps}(x) := \sup_{y\in X} \left( u(y) -\frac{1}{2 \veps} |y - x|^2 \right), \ \ x \in X.
\end{equation}
One has that $u^{\veps}$ is $\frac{2}{\veps}$-quasiconvex and decreases pointwise to $u$ as $\veps \to 0^+$. 
 There is an underlying pure second order potential theory for $\lambda$-quasiconvex functions on  $X$; namely with respect to the {\em $\lambda$-quasiconvexity} subequation 
\begin{equation}\label{LQCS}
\cP_{\lambda} := \{ A \in \cS(n): A + \lambda I \in \cP \}.
\end{equation}
 Two important results follow. 
 
	\begin{thm}[The Almost Everywhere Theorem]\label{thm:AET}
	For locally quasiconvex functions \textcolor{blue}{$u$}
	$$
	J^2_x u = (u(x), Du(x), D^2u(x)) \in \cF_x \ \ \text{for \  $\cL^n$-a.e. $x \in X$} \ \ \Longleftrightarrow \ \ u \in \cF(X).
	$$
\end{thm}
 This result is proven \cite{HL16b}. The main point in the proof is to control the measure of the set of {\em upper contact points} near $x$ if $u$ is quasiconvex. This control comes from  either of two results obtained independently by Slodkowski \cite{Sl84} and Jensen \cite{Je88}. These two measure theoretic results are shown to be equivalent in \cite{HL16a}. 

\begin{thm}[The Subharmonic Addition Theorem: Quasiconvex Version]\label{thm:SAT}
	Suppose that the subequations $\cF, \cG$ and $\cH$ satisfy
	\begin{equation*} \tag{Jet addition}
	\cF_x + \cG_x \subset \cH_x, \ \ \text{for each} \ x \in X. 
	\end{equation*}
	If $u \in \cF(X)$ and $v \in \cG(X)$ are locally quasiconvex, then
	\begin{equation*} \tag{Subharmonic addition}
	u + v \in \cH(X). 
	\end{equation*}
\end{thm}
This result appears in \cite{HL16b} and follows easily from the almost everywhere theorem. Subharmonic addition extends to $u,v \in \USC(X)$ by quasiconvex approximation in various situations. This extension has been accomplished for constant coefficient subequations $\cF$ in \cite{CHLP21}, for subequations associated to imhomogeneous pure second order equations in \cite{HL19a} and for fiberegular $\cM$-monotone subequations $\cF$ in \cite{CPR21}. 

Subharmonic addition is very important when combined with the following implication between monotonicity and jet addition 
\begin{equation}\label{jet_add}
\cF_x + \cM_x \subset \cF_x \ \ \Longrightarrow \ \ \cF_x + \wt{\cF}_x \subset\wt{\cM}_x , \ \ \text{for each} \ x \in X. 
\end{equation}
This combination has the very interesting consequence that if $\cM$ is a monotonicity cone subequation for $\cF$, then sums of $\cF$-subharmonics and $\wt{\cF}$-subharmonics are $\wt{\cM}$-subharmonics. Thus, when $\cM$ has constant coefficients, comparison for $\cF$ reduces to the validity of the zero maximum principle (ZMP) for $\wt{\cM}$-subharmonics where $\wt{\cM}$ has constant coefficients since $\cM$ does. This is a constant coefficient potential theory and has been analyzed extensively in \cite{CHLP21}, where the validity of the (ZMP) is well understood for constant coefficient monotonicity cone subequations. This will be briefly reviewed at the end of section \ref{sec:comparison} below.
	
\subsection{A ``tool kit'' for $\cF$-subharmonics (subsolutions)}

Some of the  ``nuts and bolts'' of handling $\cF$-subharmonic functions are briefly described here. The first result is both useful for checking whether a given function is $\cF$-subharmonic and also sheds light on the notion of viscosity subsolutions.

\begin{lem}[Definitional Comparison]\label{lem:DC}
Let $\cF \subset \cJ^2(X)$ be a subequation and $u \in \USC(X)$.
\begin{itemize}
	\item[(a)]	If $u$ is $\cF$subharmonic on $X$, the following form of comparison holds on any bounded domain $\Omega \subset \subset X$:
	\vspace{-5pt}
	$$
	u+v \leq 0\ \text{on}\ \partial\Omega \implies u+v \leq 0\ \text{on}\  \Omega
	$$
for each $v \in \USC(\overline{\Omega}) \cap C^2(\Omega)$ which is strictly $\wt{\cF}$-subharmonic on $\Omega$. 
	\item[(b)]	Conversely, given $u \in \USC(X)$, suppose that for each $x\in X$ there is a neighbourhood $\Omega \subset \subset X$ of $x$ where the above form of comparison holds. Then $u$ is 	$\cF$-subharmonic on $X$. 
\end{itemize}
\end{lem}

Part (a) is a well-known principle in the viscosity theory of PDEs and is here recast on the potential theory side with the help of duality. Part (b) is novel (also for its natural formulation on the operator theory side in both contrained and unconstrained cases) and shows that the {\bf {\em fundamental notion}} in viscosity theory is the validity of this form of the comparison principle. The proof of definitional comparison can be found in  \cite{CHLP21} for constant coefficient $\cF$ and in \cite{CPR21} for the general case.

The next tool is widely used to prove that a given function is $\cF$-subharmonic by a contradition argument.
\begin{lem}[Bad Test Jet Lemma] 
	Let $\cF \subset \cJ^2(X)$ be a subequation. Suppose that $u \in \USC(X)$ is \underline{not} $\cF$-subharmonic at $x_0 \in X$. Then there exist $\epsilon > 0, \rho > 0$ and a 2-jet $J \notin 	\cF_x$ such that the (unique) quadratic function $Q_J$ with $J^2_{x_0} Q_J = J$ is an upper test function for $u$ at $x$ in the following $\epsilon$-strict sense: 
	$$ 
	u(x) - Q_J(x) \leq -\epsilon|x-x_0|^2 \quad \ \forall \,  x \in B_{\rho}(x_0) \ \  \text{with equality at} \ x_0.
	$$
\end{lem} 
This is merely the contrapositve of the definition of being subharmonic in $x_0$ making use of $\epsilon$-strict upper test jets yield an equivalent definition see \cite{HL11a}. 

In addition to the {\em $C^2$-coherence} property and the {\em uniform translation} property (for continuous $\cM$-monotone $\cF$) discussed above, one has many additional properties which are useful in various constructions. 
\begin{prop}[Elementary properties of $\cF(X)$] 
	Let $\cF \subset \cJ^2(X)$ be a subequation. Then the following properties hold:	
	\begin{itemize}		
		\item[(i)] {\bf local:} \ \ $u \in \USC(X)$ is locally 
		$\cF$-subharmonic $\Leftrightarrow$ $u \in \cF(X)$;
		
		\item[(ii)] {\bf maximum:} \ \ $u,v \in \cF(X)$ $\implies$ $\max\{u,v\}	\in \cF(X)$;
		\item[(iii)] {\bf sliding:} \ \ $u\in \cF(X)$ $\implies$ $u-m \in \cF(X)$ \ \ for any $m >0$;
		\item[(iv)] {\bf decreasing limits:} \ \ $\{u_k\}_{k\in 			\mathbb{N}} \subset \cF(X)$ decreasing $\implies$ $u := \lim_{k\to\infty} u_k \in \cF(X)$;
		\item[(v)] {\bf uniform limits:} \ \ $\{u_k\}_{k\in \mathbb{N}} 	\subset \cF(X)$ locally uniformly converging to $u$ $\implies$ $u \in \cF(X)$;
		\item[(vi)] {\bf families locally bounded above:} \ \ if \, $\mathfrak{F} \subset \cF(X)$ is a non-empty family of functions which are locally uniformly bounded frpom above, then the upper semicontinuous envelope \footnote{We recall that $u^*(x):= \limsup_{r \to 0^+} \{ u(y): \ y \in X \cap \overline{B}_r(x) \}$ for each $x \in X$ and that $u^x$ is the minimal $\USC(X)$ function with $u \leq u^*$ on $X$.}  $u^*$  of the Perron function 
		$u(\, \cdot \,) := \sup_{w \in \mathfrak{F}} w(\, \cdot\, )$
		belongs to $\cF(X)$. 
	\end{itemize}
	Furthermore, if  $\cF$ has constant coefficients, the following also holds:
	\begin{itemize}
		\item[(vii)] {\bf translation:} \ \  $u \in \cF(X)$ $\iff$  $u(\cdot - y) \in \cF(X + y)$, for any  $y \in \R^n$.
	\end{itemize}	
\end{prop}	
These are familiar properties for the viscosity theory of nonlinear elliptic PDEs, although stated typically only in the unconstrained case. See \cite{HL11a} for the proofs of this general potential theoretic version. 

\section{Comparison by the monotonicity-duality-fiberegularity method}\label{sec:comparison}

In this section, we present a clean elegant and flexible method for proving {\em comparison} (the comparison principle) in nonlinear potential theory. It makes use of the two ingredients  monotonicity and duality, along with some form of regularity of the subequation. There are three incarnations of the needed regularity: constant coefficients, tameness for subequations defined by inhomogeneous equations, and fiberegularity. The method works when a given subequation admits a suitable constant coefficient monotonicity cone subequation $\cM$. When the comparison principle is combined with a correspondence principle, comparison can be transferred to nonlinear elliptic PDEs. 

\subsection{Statement and history of the general result}

The method has evolved from the constant coefficient pure second order case \cite{HL09c}. The general theorem in Euclidian space is the following result \cite[Theorem 4.3]{CPR21}.

\begin{thm}[A general comparison theorem]\label{thm:GCT}
	Let $\Omega \subset \R^n$ be a bounded domain. Suppose that a subequation $\cF \subset \cJ^2(\Omega)$ is fiberegular and $\cM$-monotone on $\Omega$ for some monotonicity cone subequation $\cM$. If $\cM$ admits a $C^2$-strict subharmonic $\psi$ on $\Omega$, then comparison holds for $\cF$ on $\overline{\Omega}$; that is,
	\begin{equation}\tag{CP}\label{cp}
	u \leq w \ \text{on $\partial \Omega$} \quad \implies \quad u \leq w \ \text{on $\Omega$}
	\end{equation}
	for all $u \in \USC(\overline{\Omega})$ which are $\cF$-subharmonic on $\Omega$, and for all $w \in \LSC(\overline{\Omega})$ which are $\cF$-superharmonic on $\Omega$.
\end{thm}

The evolution of this result can be summarized in the following way.

\noindent {\bf (1)} For $\cF_x \equiv \cF \subset \cS(n)$ (constant coefficient pure second order) see \cite{HL09c}, where $\cM = \cP$ and $\psi(x):= \frac{1}{2} |x|^2$. Here one can say 
	\begin{center}
		{\em (CP) holds for all subequations $\cF \subset \cS(n)$ and for all domains $\Omega \subset \subset \R^n$.}
	\end{center} 

\noindent {\bf (2)} For $\cF \subset \Omega \times\cS(n)$ (fiberegular variable coefficient pure second order) see \cite{CP17}, where $\cM = \cP$ and $\psi(x):= \frac{1}{2} |x|^2$. Here one can say 
		\begin{center}
			{\em (CP) holds for all fiberegular subequations $\cF \subset \Omega \times \cS(n)$ and for all domains $\Omega \subset \subset \R^n$.}
		\end{center} 
Of course an interesting case here is an inhomogeneous subequation $\cF$ defined by $F(D^2u) - f(x) \geq 0$. In \cite{HL19a}, assuming a condition called {\em tame} on the operator $F$, (CP) was established. One can show that $F$ tame implies that $\cF$ is fiberegular, so this result is a special case of (2).

\noindent {\bf (3)} For $\cF_x \equiv \cF \subset \R \times \cS(n)$ (constant coefficient gradient free) see \cite[Theorem 13.4]{HL11a} where $\cM = \cQ = \cN \times \cP$ and $\psi(x):= \frac{1}{2} (|x|^2 - R^2), R>>0$. Here one can say 
\begin{center}
	{\em (CP) holds for all subequations $\cF \subset \R \times \cS(n)$ and for all domains $\Omega \subset \subset \R^n$.}
\end{center}
 
\noindent {\bf (4)}  For $\cF \subset \Omega \times (\R \times \cS(n))$ (fiberegular variable coefficient gradient-free) see \cite{CP21}, where $\cM = \cQ = \cN \times \cP$ and $\psi(x):= \frac{1}{2} (|x|^2 - R^2), R>>0$.  Here one can say 
\begin{center}
	{\em (CP) holds for all fiberegular subequations $\cF \subset \Omega \times \R \times \cS(n)$ and for all domains $\Omega \subset \subset \R^n$.}
\end{center} 

\noindent {\bf (5)} For $\cF_x \equiv \cF \subset \cJ^2(\Omega)$ (general constant coefficients) see \cite{CHLP21}, where there is also a complete study of which cones $\textcolor{blue}{\cM}$ admit $\psi$.

\noindent {\bf (6)}
 For the general case $\cF \subset \cJ^2(\Omega)$ (fiberegular) see \cite{CPR21}, where one imports the class of admissible cones $\cM$ from the constant coefficient case.

\subsection{Outline of the proof} The main steps in the proof are the following.

\noindent {\bf Step 1:} First, use {\bf duality} to reformulate (CP) as:
\begin{equation}\tag{CP'}
u + v \leq 0 \ \text{on $\partial \Omega$} \quad \implies \quad u + v \leq 0 \ \text{on $\Omega$}
\end{equation}
for all $u \in \USC(\overline{\Omega}) \cap \cF(\Omega)$ and $v \in \USC(\overline{\Omega}) \cap \wt{\cF}(\Omega)$ (both subharmonic). Define $v :=-w$ and then the equivalence in \eqref{Vsuper3} translates to the equivalence of (CP) and (CP'). Next, note that (CP') is just the zero maximum principle (ZMP) for the sum of $\cF$ and $\wt{\cF}$ subharmonics: 
\begin{equation}\label{ZMP2}
\text{(ZMP)} \quad \quad	  \ z \leq 0 \ \text{on} \ \partial \Omega \ \ \Rightarrow z \leq 0 \ \text{on} \ \partial \Omega,
\end{equation}
$	\forall \, \textcolor{blue}{z \in \USC(\overline{\Omega}) \cap (\cF(\Omega) + \wt{\cF}(\Omega))}$. Thus it remains to prove (ZMP) in \eqref{ZMP2}.
	
\noindent {\bf Step 2 (Jet Addition):} Establish the fundamental  {\em jet addition formula} (\cite[Lemma 4.1.2]{HL13b})
\begin{equation}\label{JAF}
\cF_x + \cM_x \subset \cF_x \ \ \Longrightarrow \ \ \cF_x + \wt{\cF}_x \subset\wt{\cM}_x , \ \ \text{for each} \ x \in X. 
\end{equation}
This formula follows from elementary properties of duality and monotonicity.

\noindent {\bf Step 3:} Establish the {\em Almost Everywhere Theorem} and the quasiconvex version of the {\em Subharmonic Addition Theorem} (see Theorems \ref{thm:AET} and \ref{thm:SAT} to conclude
$$
z = u + v \in  \wt{\cM}(\Omega)
$$
if $u \in \cF(\Omega)$ and $v \in \wt{\cF}(\Omega)$ are locally quasiconvex. This difficult step relies on the Jensen or Slodkowski or Federer Lemmas.

\noindent {\bf Step 4:} Use {\bf fiberegularity} to prove the full {\em Subharmonic Addition Theorem} 
$$
		\cF(\Omega) + \wt{\cF}(\Omega) \subset \wt{\cM}(\Omega).
$$

\noindent {\bf Step 5:} Apply the following constant coefficient result \cite[Theorem 6.2]{CHLP21}.

\begin{thm}[The Zero Maximum Principle for Dual Monotonicity Cones]\label{thm:ZMP}
	Suppose that $\cM$ is a constant coefficient monotonicity cone subequation that admits a $C^2$-strict subsolution $\psi$ on a domain $\Omega \subset \subset \R^n$. Then the {\em zero maximum principle} holds for $\wt{\cM}$ on $\overline{\Omega}$; that is,
	\[\tag{ZMP}
z \leq 0 \ \text{on $\partial \Omega$} \quad \implies \quad z \leq 0 \ \text{on $\Omega$}
	\]
	for all $z \in \USC(\overline{\Omega}) \cap \wt{\cM}(\Omega)$.
\end{thm}
\begin{proof}
$\wt{\cM}$ is a (constant coefficient) subequation and hence satisfies the {\em sliding property}
	$$
	z - m \in \wt{\cM}(\Omega) \quad \text{for each} \quad 	m \in [0, +\infty).
	$$
Since $z- m < 0$ on $\partial \Omega$ compact
	$$
z - m + \veps \psi  \leq 0 \ \text{on} \ \partial \Omega \quad \text{for each} \ \veps \ \text{sufficiently small}.
	$$
 Since $z - m \in \wt{\cM}(\Omega)$ and since $\veps \psi \in C(\overline{\Omega}) \cap C^2(\Omega)$ is strictly $\cM$-subharmonic, by the {\em definitional comparison} Lemma \ref{lem:DC} (with $\cF = \wt{\cM}$ and $\wt{\cF} = \wt{\wt{\cM}} = \cM$) one has
	$$
z - m + \veps \psi  \leq 0 \ \text{on} \ \Omega \quad \text{for each} \ \veps \ \text{sufficiently small}.
	$$
	Passing to the limit for $\veps \to 0^+$, and then $m \to 0^+$ yields $z \leq 0$ on $\Omega$.
\end{proof}

The utility of the General Comparison Theorem \ref{thm:GCT} is greatly facilitated by the detailed study of monotonicity cone subequations in \cite{CHLP21}, which we briefly review.  There is a three parameter {\em fundamental family} of monotonicity cone subequations consisting of
$$
 \cM(\gamma, \cD, R):= \left\{ (r,p,A) \in \cJ^2: \ r \leq - \gamma |p|, \ p \in \cD, \ A \geq \frac{|p|}{R}I \right\}
$$
where
$$ 
\gamma \in [0, + \infty), R \in (0, +\infty] \ \text{and} \ \cD \subseteq \R^n,
$$ 
where $\cD$ is a {\em directional cone}; that is, a closed convex cone with vertex at the origin and non-empty interior (see Definition 5.2 and Remark 5.9 of \cite{CHLP21}). The family is fundamental in the sense that for any monotonicity cone subequation, there exists an element $\cM(\gamma, \cD, R)$ of the familly with $\cM(\gamma, \cD, R) \subset \cM$ (see Theorem 5.10 of \cite{CHLP21}. Hence if $\cF$ is an $\cM$-monotone subequation, then it is $\cM(\gamma, \cD, R)$-monotone for some triple $(\gamma, \cD, R)$. Moreover, from Theorem 6.3 of \cite{CHLP21}, given any element $\cM = \cM(\gamma, \cD, R)$ of the fundamental family, one knows for which domains $\Omega \subset \subset \R^n$ there is a $C^2$-strict $\cM$-subharmonic and hence for which domains $\Omega$ one has the (ZMP) for $\wt{\cM}$-subharmonics according to Theorem \ref{thm:ZMP}. There is a simple dichotomy. If $R = + \infty$, then arbitrary bounded domains $\Omega$ may be used, while in the case of $R$ finite, any $\Omega$ which is contained in a translate of the truncated cone $\cD_R := \cD \cap B_R(0)$.

\section{The correspondence principle}\label{sec:correspndence}

In this section, we discuss structural conditions on a given proper elliptic operator $F$ with domain $\cG \subset \cJ^2(X)$ which ensure that the constraint set $\cF$ defined by the compatibility relation \eqref{relation}
\begin{equation}\label{relation1}
\cF:= \{ (x,J) \in \cG: \ F(x,J) \geq 0 \}\}
\end{equation}
satisfies the two conditions needed for the  the Correspondence Principle of Theorem \ref{thm:corresp_gen}.  We recall that the these two conditions are:
\begin{equation}\label{correspondence1}
\cF \ \text{defined by} \ \eqref{relation1} \ \text{is a subequation (in the sense of Definition} \ \ref{defn:subeq})
\end{equation}
and compatibility \eqref{compatibility1} between $\cF$ and $F$:
\begin{equation}\label{correspondence2}
\Int \, \cF = \{ (x,J) \in \cG: \ F(x,J) > 0\},
\end{equation} 
or equivalently
\begin{equation}\label{correspondence3}
\partial  \cF = \{ (x,J) \in \cF: \ F(x,J) = 0\}.
\end{equation}

$\cF$ defined by \eqref{relation1} will be a subequation if it satisfies the three properties of positivity (P), negativity (N) and topological stability (T). The first two (P) and (N) are equivalent to the (fiberwise) monotnicity property that for each $x \in X$
$$
(r,p,A) \in \cF_x \ \ \Rightarrow \ \ (r + s, p, A + P) \in \cF_x, \ \ \forall \, s \leq 0 \ \text{in} \ \R, P \geq 0 \ \text{in} \ \cS(n),
$$
which clearly follows from the same monotonicity property for the domain $\cG$ and and the proper ellipticity of $F$ on $\cG$ (see \eqref{PEO}):
$$
F(x,r + s, p, A +P) \geq F(x,r,p,A), \ \ \forall (r,p,A) \in \cG_x, s \leq 0 \ \text{in} \ \R \ \text{and} \ P \geq 0 \ \text{in} \ \cS(n).
$$
This leaves the topological property (T). Recall that it requires the  three conditions
\begin{equation}\tag{T1} 
\cF = \overline{\Int \, \cF};
\end{equation}
\begin{equation}\tag{T2} 
\cF_x = \overline{\Int \, \left( \cF_x \right)}, \ \ \forall \, x \in X;
\end{equation}
\begin{equation}\tag{T3} \left( \Int \, \cF \right)_x = \Int \, \left( \cF_x \right), \ \ \forall \, x \in X.
\end{equation}
In the constant coefficient case, property (T) reduces to property (T1). In the gradient free case, one can show that property (T1) follows from properties (P) and (N) since $\cF$ is closed. In the general constant coefficient case, a sufficent condition for (T1) is that $\cF$ is closed and is $\cM$-monotone for some monotonicity cone subequation (see Proposition 4.7 of \cite{CHLP21}). $\cF$ defined by \eqref{relation1} is closed by the continuity of $F$. Hence, if $(F, \cG)$ is a constant coefficient $\cM$-monotone pair, then $\cF$ defined by \eqref{relation1} is indeed a subequation.

In the variable coefficient case, assuming that $(F, \cG)$ is an $\cM$-monotone pair, then the argument above (fiberwise) yields the property (T2). This leaves properties (T1) and (T3). It is not hard to see that if $\cF$ is closed, then properties (T2) plus (T3) imply (T1) (see Proposition A.2 of \cite{CPR21}). Hence for a $\cM$-monotone pair $(F, \cG)$, the constraint set $\cF$ defined by \eqref{relation1} will be a subequation if $\cF$ is closed and satisfies (T3). Moreover, since the inclusion $ \left( \Int \, \cF \right)_x \subset \Int \, \left( \cF_x \right)$ is automatic for each $x \in X$, (T3) reduces to the revese inlcusion, which holds provided that $\cF$ is $\cM$-monotone and fiberegular in the sense of Defintion \ref{defn:fibereg}. This fact is proved in Proposition A.5 of \cite{CPR21}. Moreover, as shown in Theorem 6.1 of \cite{CPR21}, $\cF$ will be fiberegular if $\cG$ is fiberegular provided that $F$ satisfies a mild {\em regularity condition} (see \eqref{FReg} below). In addition, fiberegularity of $\cF$ ensures that $\cF$ is closed (see Proposition A.6 of \cite{CPR21}). We collect some of these observations in the following Lemma.

\begin{lem}\label{lem:subequation} Suppose that $(F, \cG)$ is an $\cM$-monotone operator-subequation pair for some monotonicity cone subequation, with $\cG = \cJ^2(X)$ or $\cG \subsetneq \cJ^2(X)$ a fiberegular subequation. Suppose that $(F, \cG)$ satisfies the regularity condition: for some fixed $J_0 \in \Int \, \cM$, given $\Omega \subset \subset X$ and $\eta > 0$, there exists $\delta = \delta(\eta, \Omega) > 0$ such that
	\begin{equation}\label{FReg}
	F(y, J + \eta J_0) \geq F(x,J), \ \ \forall \, x,y \in \Omega \ \text{with} \ |x - y| < \delta.
	\end{equation}
Then the constraint set $\cF$ defined by \eqref{relation1} is a (fiberegular $\cM$-monotone)	subequation.
	\end{lem}

Finally, we discuss structural conditions on a proper elliptic operator $F$ with domain $\cG \subset \cJ^2(X)$ for which the constraint set $\cF$ defined by \eqref{relation1} satisfies compatibility \eqref{correspondence2} (or equivalently \eqref{correspondence3}). In the situation of Lemma \ref{lem:subequation}, which ensures that $\cF$ defined by \eqref{relation1} is a subequation, by the topological property (T3) it suffices to have \eqref{correspondence2} fiberwise; that is,
\begin{equation}
	\Int \, \cF_x = \{ J \in \cG_x: \ F(x,J) > 0\}, \ \ \forall \, x \in X.
\end{equation}
This condition is often easily checked for a given pair $(F, \cG)$ which determines $\cF$ by checking that $F(x,J) = 0$ for $J \in \partial \cF_x$ and using some strict monotonicity such as: for each $x \in X$ with some fixed $J_0 \in \Int \, \cM$ there exists $t_0 > 0$ such that
\begin{equation}\label{SM1}
	F(x,J + tJ_0) > F(x,J), \ \ \forall t \in (0, t_0), \forall \, J \in \partial \cF_x.
\end{equation}
Compatibility in this situation of a {\em homogeneous equation} $F(J^2u) = 0$ is relatively simple because one need only pay attention to $F$ in a neighborhood of the zero locus of $F$ (with domanin $\cG$). 

More structure is required if one would like to treat the {\em inhomogeneous equation}
\begin{equation}\label{InHomEq}
	F(J^2u) = \psi, \ \ \psi \in C(X)
\end{equation}
for a given constant coefficient operator $F$. This is true even for constant sources $\psi = c$. This case has been studied exensively in \cite{CHLP21}, which we now review. There the domain $\cG$ was denoted instead by $\cF$, which we will also do below. In the constrained case, where $\cF \subset \cJ^2$ is a (constant coefficient) subequation, {\bf {\em compatibility}} is defined by the two conditions 
\begin{equation}\label{CCC1}
\inf_{\cF} F \ \ \text{is finite (and denoted by $c_0$)}
\end{equation}
and
\begin{equation}\label{CCC2}
\partial \cF = \{ J \in \cF: \ F(J) = c_0 \}.
\end{equation}
Given an operator-subequation pair $(F, \cF)$, the values $c \in F(\cF)$ are called {\em admissible levels of $F$}, since otherwise the level set $\{ F = c \}$ is empty.

More is needed in order to treat the inhomogeneous equation $F(J^2u)  = c$ for all of the admissible levels. In order to avoid some obvious pathologies, one must assume that the operator $F \in C(\cF)$ is {\bf {\em topologically tame}}; that is, for each admissible level $c \in F(\cF)$, 
\begin{equation}\label{top_tame}
\text{the level set} \  \cF(c) := \{ J \in \cF: \ F(J) = c \} \ \text{has empty interior}.
	\end{equation}
This condition serves an additional purpose. Namely, if $(F, \cF)$ is a proper elliptic operator-subequation pair with $F$ topologically tame, then for every admissible level $c \in F(\cF)$ the upper level set
\begin{equation}\label{ULS}
	\cF_c := \{ J \in \cF: \ F(J) \geq c \}
\end{equation}
satisfies the topological property (T). Hence each $\cF_c$ is a subequation since properties (P) and (N) are encoded by the proper ellipticity. The obvious pathologies eliminates by topological tameness of $F$ are explained in \cite[section 11.1]{CHLP21}. For example, if some admissible level set $\cF(c)$ has non-empty interior, then one has many counterexamples for comparison by considering perturbations $v + \varphi$ of a 
local $C^2$ solution to $F(J^2v) = c$ with $\varphi$ smooth, compactly supported and with small $C^2$-norm.

Some strict monotonicity for the operator $F$ provides a convenient structual condition on the operator which eliminates such pathologies. More precisely, for constant coefficient compatible pairs $(F, \cF)$ which are $\cM$-monotone for some monotonicity cone subequation $\cM$, topological tameness \eqref{top_tame} is equivalent the following structural condition of {\bf {\em strict $\cM$-monotonicity}} on $F$:
\begin{equation}\label{SMM}
\exists \, J_0 \in \Int \, \cM \ \ \text{such that} \ \ F(J + tJ_0) > F(J)  \ \ \text{for each} \ J \in \F \ \text{and} \ t>0.
\end{equation}
In the gradient-free case this monotonicity is the weakest possible notion of being strictly proper elliptic. Moreover, these equivalent notions \eqref{top_tame} and \eqref{SMM} are also equavalent to any one of the following three conditions  (see Theorem 11.10 of \cite{CHLP21}):
\begin{itemize}
	\item[1)] $F(J + J_0) > F(J)$ for each $J \in \cF$ and each $J_0 \in \Int \,\cM$;

	\item[2)] $\{ J \in \F: F(J) > c\} = \Int \, \F_c$ for each admissible level $c \in F(\F)$;
\item[3)] $\F(c) = \F_c \cap \left( - \wt{\F}_c \right)$ for each admissible level $c \in F(\F)$.
\end{itemize}

Combining compatibility with strict $\cM$-monotonicity, one has a correspondence principle for the solutions of the inhomogeneous equation \eqref{InHomEq}, which are precisely the $\cF_{\psi}$-harmonics for the subequation with fibers
\begin{equation}\label{InHomSubEq}
	\cF_{\psi(x)} = \{ F \in \cF: \ F(J) \geq \psi(x) \}, \ \ x \in X.
\end{equation}
More precisely, one has the following result whose proof follows directly from the proof of the constant source case $\psi = c$ given in Theorem 11.13 of \cite{CHLP21}.

\begin{thm}[Correspondence principle]\label{thm:Correspondence_InHom} Suppose that $(F, \cF)$
	is a compatible $\cM$-monotone (operator-subequation) pair for some montotoncity cone subequation $\cM$ with $F$ strictly $\cM$-monotone in the sense \eqref{SMM}. Then, for any $\psi \in C(X)$ taking values in $F(\cF)$, a function $u \in C(X)$ is an $\cF$-admissibile solution of the equation $F(J^2u) = \psi$ in $X$ if and only if $u$ is $\cF_{\psi}$-harmonic in $X$. In particular, for $u \in \USC(X)$ and $w \in \LSC(X)$ one has
	$$
	\mbox{$u$ is an $\cF$-admissible subsolution of $F(J^2u) = \psi \ \Longleftrightarrow \ u$ is $\cF_{\psi}$-subarmonic}
	$$ 
	and 
	$$
	\mbox{$w$ is an $\cF$-admissible supersolution of $F(J^2u) = \psi \ \Leftrightarrow \ -w$ is $\wt{\cF}_{\psi}$-subarmonic}.
	$$
\end{thm}


\end{document}